\documentclass[11pt,reqno]{amsart}
\usepackage{times}
\usepackage[latin1]{inputenc}
\usepackage[T1]{fontenc}
\usepackage{graphics, color, amsmath, amsfonts, amssymb, amsthm, amscd, latexsym, euscript}

\usepackage[dvips]{graphicx}
\usepackage[margin=1in]{geometry}

\def\C{{\mathbb C}}

\def\R{{\mathbb R}}

\def\e{{\varepsilon}}

\def\s{{\sigma}}

\def\a{{\alpha}}

\def\Re{\mbox{Re}}

\def\+R{+_{_{ \!\! \R}}}

\def\curl{{\,\mbox{curl}\,}}

\def\hat{\widehat}
\def\wt{\widetilde}
\def\bar{\overline}

\def\S{Schr\"{o}dinger }

\DeclareMathAlphabet{\mathpzc}{OT1}{pzc}{m}{it}

\numberwithin{equation}{section}

\begin{document}

\newtheorem{theo}{Theorem}[section]
\newtheorem{pro}[theo]{Proposition}
\newtheorem{lem}[theo]{Lemma}
\newtheorem{defin}[theo]{Definition}
\newtheorem{rem}[theo]{Remark}
\newtheorem{cor}[theo]{Corollary}

\setcounter{tocdepth}{2}
\setcounter{secnumdepth}{4}

\title[Fractional NLS in one dimension]{Nonlinear fractional Schr\"{o}dinger equations in one dimension}

\author{Alexandru D. Ionescu}
\address{Princeton University}
\email{aionescu@math.princeton.edu}

\author{Fabio Pusateri}
\address{Princeton University}
\email{fabiop@math.princeton.edu}

\thanks{The first author was partially supported by a Packard Fellowship and NSF grant DMS-1065710.}

\begin{abstract} We consider the question of global existence of small, smooth, and localized solutions of a certain 
fractional semilinear cubic NLS in one dimension, 
\begin{equation*}
i \partial_t u -\Lambda u = c_0{|u|}^2 u + c_1 u^3 + c_2 u \bar{u}^2  + c_3 \bar{u}^3, \qquad \Lambda = \Lambda(\partial_x) = {|\partial_x|}^\frac{1}{2},
\end{equation*}
where $c_0\in\mathbb{R}$ and $c_1,c_2,c_3\in\mathbb{C}$. This model is motivated by the two-dimensional water waves equations, 
which have a somewhat similar structure in the Eulerian formulation, in the case of irrotational flows. 
We show that one cannot expect linear scattering, even in this simplified model. 
More precisely, we identify a suitable nonlinear logarithmic correction, and prove global existence and modified scattering of solutions. 
\end{abstract}

\maketitle

\tableofcontents

\section{Introduction}
We consider the Cauchy problem for a class of fractional nonlinear \S (NLS) equations in dimension one with cubic nonlinearities:
\begin{equation}\label{model}
i \partial_t u -\Lambda u = c_0{|u|}^2 u + c_1 u^3 + c_2 u \bar{u}^2  + c_3 \bar{u}^3, \qquad \Lambda = \Lambda(\partial_x) = {|\partial_x|}^\frac{1}{2},
\end{equation}
where $u: \R_t \times \R_x \rightarrow \C$, $c_0\in\mathbb{R}$, and $c_1,c_2,c_3 \in \C$. This model is motivated by the question of global existence of solutions of the two-dimensional water wave equation, see subsection \ref{WWmot} for a longer discussion.

We are interested in the Cauchy problem for small initial data $ { u(t,x) |}_{t=0} := u_0(x)$ given in a suitable weighted Sobolev space.
We investigate the global existence and long time behaviour of solutions to \eqref{model}. More precisely, we prove the following:

\begin{theo}\label{maintheo}
Assume that $N_0:=100$, $p_0\in(0,1/1000]$ is fixed, and $u_0\in H^{N_0}(\mathbb{R})$ satisfies
\begin{equation}\label{mainhyp}
\|u_0\|_{H^{N_0}}+\|x\cdot\partial u_0\|_{L^2}+\|(1+|\xi|)^{10}\widehat{u_0}(\xi)\|_{L^\infty_\xi}=\varepsilon_0\leq\overline{\varepsilon},
\end{equation}
for some constant $\overline{\varepsilon}$ sufficiently small (depending only on the value of $p_0$). Then there is a unique global solution 
$u\in C([0,\infty):H^{N_0}(\mathbb{R}))$ of the initial-value problem
\begin{equation}\label{IVP0}
i \partial_t u -\Lambda u =c_0 {|u|}^2 u + c_1 u^3 + c_2 u \bar{u}^2  + c_3 \bar{u}^3, \qquad u(0)=u_0.
\end{equation}
In addition, letting $f(t):=e^{it\Lambda}u(t)$, we have the uniform bounds
\begin{equation}\label{mainconcl1}
\sup_{t\in[0,\infty)}\big[(1+t)^{-p_0}\|f(t)\|_{H^{N_0}}+(1+t)^{-p_0}\|x\cdot (\partial f)(t)\|_{L^2}+
\|\,(1+|\xi|)^{10}\widehat{f}(\xi,t)\|_{L^\infty_\xi}\big]\lesssim\varepsilon_0.
\end{equation}
Furthermore the solution possesses the following modified scattering behavior: there is $p_1>0$ and $w_\infty\in L^\infty$ with the property that
\begin{equation}\label{mainconcl2}
\sup_{t\in[0,\infty)}(1+t)^{p_1}{ \left\|   \exp \left( i \frac{2c_0}{\pi}\int_0^t {|\xi|}^{3/2} {\left| \hat{f}(s,\xi) \right|}^2 \frac{ds}{s+1} \right) 
      (1+|\xi|)^{10} \widehat{f}(\xi,t)  -  w_\infty(\xi)  \right\|}_{L^\infty_\xi}\lesssim \varepsilon_0.
\end{equation}
\end{theo}

{\bf{Remark:}} We emphasize that it is important to identify the correct logarithmic correction that describes 
the asymptotic behavior of solutions in \eqref{mainconcl2}, even if one is only interested in the question of global 
existence of smooth solutions. Without identifying such a logarithmic correction, it seems that one could only prove almost global existence, i.e. with a time of existence 
$T\approx e^{c/\varepsilon_0}$. This is consistent with the almost global existence result of Wu \cite{WuAG}, in the case of the irrotational two-dimensional water wave problem.  

\subsection{Previous results on modified scattering}
There is a large amount of literature dealing with the problem of global existence and asymptotic behavior of small solutions of 
nonlinear dispersive PDEs. Some key developments include the work of John \cite{Joh} showing that blow-up in finite time {\it can} happen even 
for small smooth localized initial data of a semilinear wave equation, the introduction of the vector field method by Klainerman \cite{Kl} 
and of the normal form transformation by Shatah \cite{Sh}, and the understanding of the role of "null structures", starting with the works of 
Klainerman \cite{Kl2} and Christodoulou \cite{Ch}. One of the main objective is to show that solutions evolving from small, sufficiently regular and localized data, 
behave like solutions to the linear equation. 

If the effects of the nonlinearity become negligible when time tends to infinity, solutions are said to scatter to a linear 
asymptotic state. However, there are several important examples of equations whose small solutions do not behave like linear ones,
as it is the case for the fractional \S equation \eqref{model}.
In what follows we give a brief account of some previous results concerning the nonlinear \S equation, which is the most closely related to our problem,
and a few other dispersive equations. 
We will then point out some important connections between \eqref{model} and the water waves system in $2$ dimensions.

Let us start by considering the \S equation
\begin{equation}
\label{NLS}
i \partial_t u + \Delta u = N(u,\bar{u}) ,
\end{equation}
where $u : \R_t \times \R_x^d \rightarrow \C$,
and $N$ is a nonlinear function of $u$ and its conjugate $\bar{u}$.
For $N = {|u|}^{p-1} u$ one distinguishes the short range case $p > 1 + \frac{2}{n}$ and the long range case $p \leq 1 + \frac{2}{n}$.
A simple explanation for this distinction is the fact that the nonlinearity computed on a linear solution
is integrable in time in the short range case, whereas it is not integrable in the long range case.
In the short range case wave operators can be contructed in general for small data \cite{GOVNLS,NakaNLS}.
The situation is quite different in the long range case, where it is known since \cite{Barab} 
that in one dimension nontrivial asymptotically free solutions cannot exist.
Ozawa \cite{ozawa} showed that long range scattering (i.e scattering to a nonlinear profile) occurs in the critical case $p=3$ in one dimension. 
Hayashi and Naumkin \cite{HN} showed the same result in two and three dimension,
and also in the case of the Hartree equation in $d \geq 2$.
%
%
A different proof for the $1$d NLS and the Hartree equations was given by the second author and Kato in \cite{KP}.
We point out here that one of the key ingredients in \cite{HN} is an explicit factorization of the linear \S semigroup,
which may not be available in the case of other equations, such as the one considered in this paper.
As shown in \cite{KP}, a ``stationary phase'' type argument, inspired by the Fourier analysis of \cite{GMS1},
can serve as a substitute for such a factorization. 
This type of argument is going to be an important ingredient in the proof of Theorem \ref{maintheo}.

The problem for \eqref{NLS} with general cubic nonlinearities  has also been studied extensively.
For the same nonlinearity as in \eqref{model}, global solutions, 
again possessing a modified asymptotic behavior, were constructed in \cite{HNcubicNLSodd} for odd initial data\footnote{We refer the reader 
to the  works referenced in the introduction of \cite{HNcubicNLSodd} 
for more results about the long time behavior of solutions to cubic NLS equations with non-gauge invariant nonlinearities in one dimension.}.
Works concerning other dispersive equations, which address the existence of small solutions for long range nonlinearities,
and in particular the question of modified scattering, include \cite{HNKdV,HNBO,DelortKG1d}. 



\subsection{Motivation: Water Waves in two dimensions}\label{WWmot}
Our main interest in the fractional NLS model \eqref{model} comes from the study of the long-time behavior
of solutions to the water waves equations on $\R^2$.
In particular, as we shall describe below, $\Lambda = {|\partial_x|}^{1/2}$ is the dispersion relation of the 
linearized gravity water waves equations for one dimensional surfaces. 
Furthermore, thanks to the absence of resonances at the quadratic level, one expects the nonlinear dynamics 
of water waves to be governed by nonlinearities of cubic type\footnote{
This is indeed the case in three space dimensions \cite{GMS2,Wu3DWW}.} 
like those appearing in our model.

The evolution of an inviscid perfect fluid that occupies a domain $\Omega_t$ in $\R^n$ ($n \geq 2$) at time $t$,
is described by the free boundary incompressible Euler equations.
If $v$ and $p$ denote respectively the velocity and the pressure of the fluid (which is assumed to have constant density equal to $1$),
these equations are:
\begin{equation}
\tag{E}
\label{E}
\left\{
\begin{array}{ll}
(v_t + v \cdot \nabla v) = - \nabla p - g e_n  &   x \in \Omega_t
\\
\nabla \cdot v = 0    &   x \in \Omega_t
\\
v (0,x) = v^0 (x)     &   x \in \Omega_0 \, , 
\end{array}
\right.
\end{equation}
where $g$ is the gravitational constant. 
The free surface $S_t := \partial \Omega_t$ moves with the normal component of the velocity,
and, in absence of surface tension, the pressure vanishes on the boundary:
\begin{equation}
\tag{BC}
\label{BC}
\left\{
\begin{array}{l}
\partial_t + v \cdot \nabla  \,\, \mbox{is tangent to} \,\, \bigcup_t S_t \subset \R^{n+1}
\\
p (t,x) = 0  \,\, , \,\,\, x \in S_t  \, .
\end{array}
\right.
\end{equation}
Following the breakthrough of Wu \cite{Wuloc1,Wuloc2}
who showed wellposedness for data in Sobolev spaces for the irrotational problem ($\curl v = 0$)
with infinite depth, there has been considerable amount of work on the local well-posedness of \eqref{E}-\eqref{BC}.
See for example \cite{Lindblad,Lannes,ShZ3}. See also references therein for earlier works on this problem.

In the case of irrotational flows one can reduce \eqref{E}-\eqref{BC} to a system on the boundary.
Assume that $\Omega_t \subset \R^2$ is the region below the graph of a function $h : \R_x \times \R_t \rightarrow \R$,
that is $\Omega_t = \{ (x,y) \in \R^2 \, : y \leq h(x,t) \}$.
Let us denote by $\Phi$ the velocity potential: $\nabla \Phi(t,x,y) = v (t,x,y)$, for $(x,y) \in \Omega_t$.
If $\phi(t,x) := \Phi (t, x, h(x,t))$ is the restriction of $\Phi$ to the boundary $S_t$,
the equations of motion reduce to the system\footnote{This nontrivial rewriting of the equations
is based upon an expansions of the Dirichlet to Neumann operator associated to the domain $\Omega_t$ for small
perturbations of a flat surface. Here we are taking $g = 1$.}
\cite{SulemBook}
\begin{equation}
\label{Euler}
\left\{
\begin{array}{rl}
\partial_t h  & = {|\partial_x|} \phi  - \partial_x( h  \partial_x \phi) - {|\partial_x|} (h {|\partial_x|} \phi)
	  \\
	  & - \frac{1}{2} {|\partial_x|} \left[  h^2 {|\partial_x|}^2 \phi + {|\partial_x|} (h^2 {|\partial_x|} \phi)
	  - 2 (h {|\partial_x|} (h {|\partial_x|} \phi ) )  \right] 
	  + R_1
\\
\partial_t \phi  & =  - h - \frac{1}{2} {|\phi_x|}^2 + \frac{1}{2} {|{|\partial_x|} \phi|}^2 + 
	  {|\partial_x|} \phi \left[ h {|\partial_x|}^2 \phi - {|\partial_x|} (h {|\partial_x|} \phi) \right]
	  + R_2
\end{array}
\right.
\end{equation}
where $R_1$ and $R_2$ are terms of order $4$ or higher.
Defining $u := h + i \Lambda \phi$, \eqref{Euler} can be reduced to a scalar equation of the form
\begin{equation}
\label{WWmodel}
i \partial_t u - \Lambda u = Q (u,\bar{u}) + C (u,\bar{u}) + R (u,\bar{u}) ,
\end{equation}
where $Q$ is a quadratic form of $u$ and $\bar{u}$, $C$ denotes cubic terms and $R$ denotes quartic and higher order terms.
$Q,C$ and $R$ in \eqref{WWmodel} are of course determined by the nonlinearities in \eqref{Euler}.
We refer to  \cite[chap. 11]{SulemBook} for the derivation of the  water wave equations
and to \cite[sec. 3]{GMS2} for the explicit form of \eqref{WWmodel}.

Unlike our model \eqref{model}, the water waves equations \eqref{WWmodel} contain quadratic terms.
Since the pointwise decay of a linear solution is $t^{-1/2}$, quadratic terms are far from having integrable-in-time $L^2$ norm,
and this makes \eqref{WWmodel} supercritical with respect to scattering.
On the other hand, it is well known, see for example \cite{Craig,CW,GMS2}, 
that the gravity water waves equations present no quadratic resonances. 
This allows to find a bilinear change of variables $v = u + B(u,u)$, such that the new unknown $v$ satisfies an equation of the form
\begin{equation}
\label{WWcubicmodel}
i \partial_t v - \Lambda v = \wt{C} (v,\bar{v}) +  \wt{R} (u,\bar{u}) ,
\end{equation}
where $\wt{C}$ is a cubic nonlinearity in $v$ and $\bar{v}$, and $\wt{R}$ denotes quartic and higher order terms. This normal form transformation 
eliminating the quadratic terms, plays a crucial role in \cite{GMS2} where the authors obtain global existence of small solutions
to the gravity water waves equations in three space dimension, i.e. in the case of two dimensional surfaces.

While both \eqref{model} and \eqref{WWcubicmodel} have cubic nonlinearities,
it is important to remark that the nonlinearity in \eqref{WWcubicmodel}, as well as that of \eqref{WWmodel}, contains derivatives of the unknown.
This fact poses great additional difficulty in both the local and global Cauchy theory for the water waves system.
Equation \eqref{model} admits straightforward energy estimates, but it is not at all clear whether \eqref{WWcubicmodel} 
does as well, at least in the basic Eulerian formulation described above.
As far as dispersive estimates are concerned, such a difficulty can be overcome fairly easily in the case of two dimensional surfaces \cite{GMS2}, 
since the decay of linear solutions is $t^{-1}$, and energy estimates can be proven separately via a different fomulation of the equations \cite{ShZ3}.
We refer the reader also to the work of Wu \cite{Wu3DWW} for a different proof of the global existence of solutions to gravity water waves in 3D.

In the case of $1$ dimensional surfaces, it is not known whether global solutions exist.
The only work investigating the long time behavior of small solutions is the paper of Wu \cite{WuAG}, who obtained almost global existence.
In \cite{WuAG}, as well as in \cite{Wu3DWW} and \cite{WuNLS}, 
a nonlinear version of a normal form transformation is used in order to recast the quadratic equations into cubic ones.
Remarkably, the cubic equations obtained by Wu in \cite{WuAG} admit energy (resp. weighted energy) estimates in Sobolev (resp. weighted Sobolev) spaces, 
unlike the cubic equations \eqref{WWcubicmodel} obtained in \cite{GMS2}.
However, the energy estimates in \cite{WuAG} are not optimal, 
and can be used only to obtain decay estimates on time scales of the order $e^{c/\e_0}$, where $\e_0$ is the size of the initial data.
Also, the formulation in \cite{WuAG} does not seem to be well-suited for the type of Fourier analysis performed in this paper.

We propose here to analyze \eqref{model} as a simplified model for the leading order cubic dynamics in the $2$D water waves equations,
as given by \eqref{WWcubicmodel}.
Theorem \ref{maintheo} shows that \eqref{model} admits global solutions whose long time behavior is not linear.
In particular, a correction of logarythmic type, see \eqref{mainconcl2}, is needed in order to obtain the $t^{-1/2}$ decay 
and the scattering of solutions. We emphasize that having a precise understanding of this correction is a key component of the 
global-in-time analysis.

As already pointed out, the advantage in the analysis of \eqref{model} 
lies in the fact that the symbol of the nonlinear interaction is just taken to be $1$,
so that the difficulty concerning the energy and $L^2$-based estimates does not enter the problem\footnote{
We note however that some of the structure in the nonlinearity of the water wave equation that could be of help is disregarded by doing so.}.
Nevertheless, as far as the global-in-time pointwise behavior of solutions to the $2$D gravity water waves is concerned, 
\eqref{model} can be considered an appropriate 
model.


We conclude by mentioning that in the physics literature the fractional \S equation was introduced by Laskin \cite{Laskin2} 
in deriving a fractional version of the classical quantum mechanics.
For the nonlinear cubic gauge invariant equation, with dispersion ${|\partial_x|}^\alpha$ for $1 < \alpha < 2$,
global existence for $L^2$ data was obtained in \cite{GuoHuo}, combining multilinear estimates based on Bourgain spaces with mass conservation.
It would be interesting to see whether our global existence and modified scattering result  
can be generalized to other fractional powers $0 < \a < 2$ with $\a \neq 1$. Of particular appeal would be the case $\a = \frac{3}{2}$, 
given its possible relevance to the $1$ dimensional water waves equations with surface tension (capillary waves).

\vskip10pt
Our paper is organized as follows. In section \ref{maintheo} we prove Theorem \ref{maintheo} as a consequence of a bootstrap 
argument based on the local existence theory (Proposition \ref{locTh}), on a refined linear dispersive estimate (Lemma \ref{dispersive}),
and on a priori estimates in a suitably constructed space (Proposition \ref{bootstrap}).
We then proceed to prove Lemma \ref{dispersive} in section \ref{proof1}.
Proposition \ref{bootstrap} follows as a consequence of Propositions \ref{bigbound1} and \ref{bigbound2}.
The proof of this latter constitutes the most technical part of the paper and is performed in section \ref{proof2.2}.

\section{Proof of Theorem \ref{maintheo}}

We define the normed spaces
\begin{equation}\label{def}
\begin{split}
&W:=\{f\in H^4(\mathbb{R}):\|f\|_{W}:=\|f\|_{L^2}+\|x\cdot\partial f\|_{L^2}<\infty\},\\
&Z:=\{f\in H^4(\mathbb{R}):\|f\|_{Z}:=\|\,(1+|\xi|)^{10}\widehat{f}(\xi)\|_{L^\infty_\xi}<\infty\}.
\end{split}
\end{equation}

We start with the local theory:

\begin{pro}\label{locTh}

(i) Given $u_0\in H^4(\mathbb{R})$ there is $T_0=T_0(\|u_0\|_{H^4})>0$ and a unique solution $u\in C([-T_0,T_0]:H^4)$ of the initial-value problem
\begin{equation}\label{IVP}
i \partial_t u -\Lambda u = c_0{|u|}^2 u + c_1 u^3 + c_2 u \bar{u}^2  + c_3 \bar{u}^3, \qquad u(0)=u_0.
\end{equation}

(ii) Assume $N\geq 4$ and $u_0\in H^N(\mathbb{R})$, and let $u\in C([-T_0,T_0]:H^4)$, $T_0=T_0(\|u_0\|_{H^4})$, denote the solution constructed in part (i). 
Then $u\in C([-T_0,T_0]:H^N)$ and
\begin{equation}\label{loc1}
\|u(t_2)\|_{H^N}-\|u(t_1)\|_{H^N}\lesssim_N \int_{t_1}^{t_2}\|u(s)\|_{H^N}\|u(s)\|_{L^\infty}^2\,ds
\end{equation}
for any $t_1\leq t_2\in[-T_0,T_0]$.
\end{pro}

\begin{proof}[Proof of Proposition \ref{locTh}] The proposition follows from a standard fixed-point argument: the solution $u$ is constructed as the unique
solution in the complete metric space $X:=\{v\in C([-T_0,T_0]:H^4):\sup_{t\in[-T_0,T_0]}\|v(t)\|_{H^4}\leq 2\|u_0\|_{H^4}\}$ of the equation
\begin{equation*}
 u(t)=e^{-it\Lambda}u_0-i\int_0^te^{-i(t-s)\Lambda}[c_0\bar{u}(s) u(s)^2 + c_1 u(s)^3 + c_2 u(s) \bar{u}(s)^2  + c_3 \bar{u}(s)^3]\,ds.
\end{equation*}
The inequality \eqref{loc1} follows from this definition as well. See also the proof of Lemma \ref{bigbound1} below for the complete details in a more complicated situation.
\end{proof}

Our main ingredient is the following bootstrap estimate.

\begin{pro}\label{bootstrap}
Assume that $N_0=100$, $T>0$ and assume that $u\in C([0,T]:H^{N_0})$ is a solution of the initial-value problem
\begin{equation}\label{IVP2}
i \partial_t u -\Lambda u = c_0{|u|}^2 u + c_1 u^3 + c_2 u \bar{u}^2  + c_3 \bar{u}^3, \qquad u(0)=u_0,
\end{equation}
with the property that
\begin{equation}\label{loc3}
\|u_0\|_{H^{N_0}}+\|u_0\|_{W}+\|u_0\|_{Z}=\varepsilon_0\leq\overline{\varepsilon}.
\end{equation}
Let $f(t):=e^{it\Lambda}u(t)$, $t\in[0,T]$. 

(i) The mapping $t\to f(t)$ is a continuous mapping from $[0,T]$ to $Z\cap W$.

(ii) Assume, in addition, that $p_0\in(0,1/1000]$ and
\begin{equation}\label{loc4}
\sup_{t\in[0,T]}\big[(1+t)^{-p_0}\|f(t)\|_{H^{N_0}}+(1+t)^{-p_0}\|f(t)\|_{W}+\|f(t)\|_{Z}\big]\leq\varepsilon_1,
\end{equation}
for some $\varepsilon_1\in[\varepsilon_0,1]$. Then
\begin{equation}\label{loc5}
 \sup_{t\in[0,T]}\big[(1+t)^{-p_0}\|f(t)\|_{W}+\|f(t)\|_{Z}\big]\leq 2\varepsilon_0+C_{p_0}\varepsilon_1^2,
\end{equation}
for some constant $C_{p_0}$ that may depend only on the exponent $p_0$.

(iii) Assume that \eqref{loc4} holds and let
\begin{equation*}
H(\xi,t):=\frac{2c_0}{\pi}|\xi|^{3/2}\int_0^t|\widehat{f}(\xi,s)|^2 \frac{ds}{s+1},\qquad t\in[0,T].
\end{equation*}
Then there is $p_1>0$ such that
\begin{equation}\label{loc5.5}
(1+t_1)^{p_1}\Big\|(1+|\xi|)^{10}\big[e^{iH(\xi,t_2)}\widehat{f}(\xi,t_2)-e^{iH(\xi,t_1)}\widehat{f}(\xi,t_1)\big]\Big\|_{L^\infty_\xi}\lesssim \varepsilon_1^2.
\end{equation}
 for any $t_1\leq t_2\in [0,T]$.
\end{pro}

The last ingredient in the proof of Theorem \ref{maintheo} is the following dispersive linear estimate:

\begin{lem}\label{dispersive}
 For any $t\in\mathbb{R}$ we have
\begin{equation}\label{disperse}
 \|e^{i t\Lambda}f\|_{L^\infty}\lesssim (1+|t|)^{-1/2}\|\,|\xi|^{3/4}\widehat{f}(\xi)\|_{L^\infty_\xi}+(1+|t|)^{-5/8}\big[\|x\cdot\partial f\|_{L^2}+\|f\|_{H^2}\big].
\end{equation}
\end{lem}

We prove Proposition \ref{bootstrap} in section \ref{proof2} and we prove Lemma \ref{dispersive} in section \ref{proof1}. 
In the rest of this section we show how to combine these ingredients to complete the proof of the main theorem.

\begin{proof}[Proof of Theorem \ref{maintheo}] Assume we are given data $u_0$ satisfying \eqref{mainhyp}, i.e.
\begin{equation*}
 \|u_0\|_{H^{N_0}}+\|u_0\|_{W}+\|u_0\|_{Z}=\varepsilon_0\leq\overline{\varepsilon},
\end{equation*}
for some $\overline{\varepsilon}$ sufficiently small. In view of Proposition \ref{locTh} there is $T>0$ and a solution $u\in C([0,T]:H^{N_0})$ of the initial-value 
problem \eqref{IVP0} with the property that if $f(t)=e^{it\Lambda}u(t)$ then
\begin{equation}\label{loc6}
\sup_{t\in[0,T]}(1+t)^{-p_0}\|f(t)\|_{H^{N_0}}\leq\varepsilon_0^{3/4}. 
\end{equation}
In view of Proposition \ref{bootstrap} (i), the mapping $t\to f(t)$ is a continuous mapping from $[0,T]$ to $Z\cap W$. Let $T'$ denote the largest number in $[0,T]$ 
with the property that
\begin{equation*}
 \sup_{t\in[0,T']}\big[(1+t)^{-p_0}\|f(t)\|_{W}+\|f(t)\|_{Z}\big]\leq\varepsilon_0^{3/4}. 
\end{equation*}
Using now Proposition \ref{bootstrap} (ii), it follows that $\sup_{t\in[0,T']}\big[(1+t)^{-p_0}\|f(t)\|_{W}+\|f(t)\|_{Z}\big]\leq 3\varepsilon_0$. 
Therefore $T'=T$ and
\begin{equation}\label{loc7}
 \sup_{t\in[0,T]}\big[(1+t)^{-p_0}\|f(t)\|_{W}+\|f(t)\|_{Z}\big]\leq 3\varepsilon_0.  
\end{equation}

We observe now that $u(t)=e^{-it\Lambda}f(t)$. Using Lemma \ref{dispersive} and \eqref{loc6}--\eqref{loc7}, it follows that
\begin{equation*}
\|u(t)\|_{L^\infty}\lesssim \varepsilon_0^{3/4}(1+t)^{-1/2},\qquad\text{ for any }t\in[0,T].
\end{equation*}
Letting $P(t):=\|f(t)\|_{H^{N_0}}$, it follows from \eqref{loc1} that
\begin{equation*}
P(t)-P(0)\lesssim \varepsilon_0^{3/4}\int_{0}^tP(s)(1+s)^{-1}\,ds
\end{equation*}
for any $t\in[0,T]$. Therefore $P(t)\lesssim P(0)(1+t)^{p_0}$ for any $t\in[0,T]$, i.e.
\begin{equation}\label{loc10}
\sup_{t\in[0,T]}(1+t)^{-p_0}\|f(t)\|_{H^{N_0}}\lesssim \varepsilon_0.
\end{equation}
As a consequence, if $u\in C([0,T]:H^{N_0})$ is a solution that satisfies the weaker bound \eqref{loc6}, then $u$ has to satisfy the stronger bound \eqref{loc10}.
Therefore, the solution $u$ can be extended to the full interval $[0,\infty)$, and the desired bound \eqref{mainconcl1} follows from \eqref{loc7} and \eqref{loc10}. 

The modified scattering behaviour \eqref{mainconcl2} is a consequence of Proposition \ref{bootstrap} (iii). This completes the proof of the theorem.
\end{proof}

\section{Proof of Lemma \ref{dispersive}}\label{proof1}

In this section we prove Lemma \ref{dispersive}. We fix $\varphi:\mathbb{R}\to[0,1]$ an even smooth function supported in $[-8/5,8/5]$ and 
equal to $1$ in $[-5/4,5/4]$. Let
\begin{equation*}
\varphi_k(x):=\varphi(x/2^k)-\varphi(x/2^{k-1}),\qquad k\in\mathbb{Z},\,x\in\mathbb{R}.
\end{equation*}
More generally, for any $m,k\in\mathbb{Z}$, $m\leq k$, we define
\begin{equation}\label{disp0}
 \varphi^{(m)}_k(x):=
\begin{cases}
\varphi(x/2^k)-\varphi(x/2^{k-1}),\qquad &\text{ if }k\geq m+1,\\
\varphi(x/2^k),\qquad &\text{ if }k=m.
\end{cases}
\end{equation}
For any interval $I\subseteq\mathbb{R}$ we define
\begin{equation}\label{disp0.5}
\varphi_I:=\sum_{k\in I\cap\mathbb{Z}}\varphi_k,\qquad \varphi^{(m)}_I:=\sum_{k\in I\cap\mathbb{Z}\cap[m,\infty)}\varphi^{(m)}_k.
\end{equation}

Let $P_k$, $k\in\mathbb{Z}$, denote the operator on $\mathbb{R}$ defined by the Fourier multiplier $\xi\to \varphi_k(\xi)$. For \eqref{disperse} it suffices to prove
that
\begin{equation}\label{disp1}
 \sum_{k\in\mathbb{Z}}\Big|\int_{\mathbb{R}}e^{it\Lambda(\xi)}e^{ix\xi}\widehat{f}(\xi)\varphi_k(\xi)\,d\xi\Big|\lesssim 1,
\end{equation}
for any $t,x\in\mathbb{R}$ and any function $f$ satisfying
 \begin{equation}\label{disp2}
  (1+|t|)^{-1/2}\|\,|\xi|^{3/4}\widehat{f}(\xi)\|_{L^\infty_\xi}+(1+|t|)^{-5/8}\big[\|x\cdot\partial f\|_{L^2}+\|f\|_{H^2}\big]\leq 1.
 \end{equation}

Using only the bound $\|f\|_{H^2}\lesssim (1+|t|)^{5/8}$, we estimate first the contribution of small frequencies, 
\begin{equation*}
 \sum_{2^k\leq 2^{10}(1+|t|)^{-5/4}}\Big|\int_{\mathbb{R}}e^{it\Lambda(\xi)}e^{ix\xi}\widehat{f}(\xi)\varphi_k(\xi)\,d\xi\Big|
\lesssim \sum_{2^k\leq 2^{10}(1+|t|)^{-5/4}}2^{k/2}\|\widehat{P_kf}\|_{L^2}\lesssim 1,
\end{equation*}
and the contribution of large frequencies,
\begin{equation*}
 \sum_{2^k\geq 2^{-10}(1+|t|)}\Big|\int_{\mathbb{R}}e^{it\Lambda(\xi)}e^{ix\xi}\widehat{f}(\xi)\varphi_k(\xi)\,d\xi\Big|
\lesssim \sum_{2^k\geq 2^{-10}(1+|t|)}2^{k/2}\|\widehat{P_kf}\|_{L^2}\lesssim 1.
\end{equation*}
Therefore, for \eqref{disp1} it suffices to prove that
\begin{equation}\label{disp4}
 \sum_{2^{10}(1+|t|)^{-5/4}\leq 2^k\leq 2^{-10}(1+|t|)}\Big|\int_{\mathbb{R}}e^{it\Lambda(\xi)}e^{ix\xi}\widehat{f}(\xi)\varphi_k(\xi)\,d\xi\Big|\lesssim 1.
\end{equation}

In proving \eqref{disp4} we may assume that $|t|\geq 1$. We estimate first the nonstationary contributions. Using \eqref{disp2} we see that 
$\|\widehat{P_kf}\|_{L^2}+2^k\|\partial(\widehat{P_kf})\|_{L^2}\lesssim |t|^{5/8}$. Therefore, if $2^{-k/2+4}\leq |x/t|$ or $|x/t|\leq 2^{-k/2-4}$ 
then we integrate by parts to estimate
\begin{equation*}
 \Big|\int_{\mathbb{R}}e^{it\Lambda(\xi)}e^{ix\xi}\widehat{P_kf}(\xi)\,d\xi\Big|
\lesssim |t|^{-1}2^{k/2}\cdot \|\partial(\widehat{P_kf})\|_{L^1}+|t|^{-1}2^{-k/2}\cdot \|\widehat{P_kf}\|_{L^1}\lesssim |t|^{-3/8}.
\end{equation*}
Therefore, for \eqref{disp4} it suffices to prove that
\begin{equation}\label{disp5}
\Big|\int_{\mathbb{R}}e^{it\Lambda(\xi)}e^{ix\xi}\widehat{f}(\xi)\varphi_k(\xi)\,d\xi\Big|\lesssim 1,
\end{equation}
provided that $|t|\geq 1$ and $2^k\in[2^{10}(1+|t|)^{-5/4},2^{-10}(1+|t|)]\cap [2^{-8}t^2/x^2,2^8t^2/x^2]$.

Let $\Psi(\xi):=t\Lambda(\xi)+x\xi$ and notice that $|\Psi''(\xi)|\approx |t||\xi|^{-3/2}$. Let $\xi_0\in\mathbb{R}$ denote the unique solution of the equation 
$\Psi'(\xi)=0$, i.e.
\begin{equation*}
\xi_0:=\mathrm{sign}(t/2x)\frac{t^2}{4x^2}.
\end{equation*}
Clearly, $|\xi_0|\approx 2^k$. Let $l_0$ denote the smallest integer with the property that $2^{2l_0}\geq 2^{3k/2}|t|^{-1}$ and estimate the 
left-hand side of \eqref{disp5} by
\begin{equation}\label{disp6}
\Big|\int_{\mathbb{R}}e^{it\Lambda(\xi)}e^{ix\xi}\widehat{f}(\xi)\varphi_k(\xi)\,d\xi\Big|\leq\sum_{l=l_0}^{k+100}|J_l|,
\end{equation}
where, with the notation in \eqref{disp0}, for any $l\geq l_0$, 
\begin{equation*}
J_{l}:=\int_{\mathbb{R}}e^{i\Psi(\xi)}\cdot \widehat{P_kf}(\xi)\varphi_l^{(l_0)}(\xi-\xi_0)\,d\xi.
\end{equation*}
It follows from \eqref{disp2} that
\begin{equation*}
\|\widehat{P_kf}\|_{L^\infty}\lesssim |t|^{1/2}2^{-3k/4},\qquad \|\widehat{P_kf}\|_{L^2}+2^k\|\partial(\widehat{P_kf})\|_{L^2}\lesssim |t|^{5/8}.
\end{equation*}
Therefore
\begin{equation*}
 |J_{l_0}|\lesssim 2^{l_0}\|\widehat{P_kf}\|_{L^\infty}\lesssim 2^{3k/4}|t|^{-1/2}\cdot |t|^{1/2}2^{-3k/4}\lesssim 1.
\end{equation*}
Moreover, since $|\Psi'(\xi)|\gtrsim |t|2^{-3k/2}2^l$ whenever $|\xi|\approx 2^k$ and $|\xi-\xi_0|\approx 2^l$, we can integrate by parts to estimate
\begin{equation*}
\begin{split}
|J_l|&\lesssim \frac{1}{|t|2^{-3k/2}2^l}\big[2^{-l}\|\widehat{P_kf}(\xi)\cdot\mathbf{1}_{[0,2^{l+4}]}(|\xi-\xi_0|)\|_{L^1_\xi}+
\|\partial(\widehat{P_kf})(\xi)\cdot\mathbf{1}_{[0,2^{l+4}]}(|\xi-\xi_0|)\|_{L^1_\xi}\big]\\
&\lesssim |t|^{-1}2^{3k/2}2^{-l}\big[\|\widehat{P_kf}\|_{L^\infty_\xi}+
2^{l/2}\|\partial(\widehat{P_kf})\|_{L^2}\big]\\
&\lesssim |t|^{-1/2}2^{3k/4}2^{-l}+|t|^{-3/8}2^{k/2}2^{-l/2}.
\end{split}
\end{equation*}
The desired bound \eqref{disp5} follows from \eqref{disp6} and the last two estimates. This completes the proof of the lemma.

\section{Proof of Proposition \ref{bootstrap}}\label{proof2}

It follows from the definitions that
\begin{equation}\label{bn1}
\begin{split}
&(\partial_t\widehat{f})(\xi,t)=(-i)(2\pi)^{-2}[c_0I_0(\xi,t)+c_1I_1(\xi,t)+c_2I_2(\xi,t)+c_3I_3(\xi,t)],\\
&I_0(\xi,t):=\int_{\mathbb{R}\times\mathbb{R}}e^{it[\Lambda(\xi)-\Lambda(\xi-\eta)-\Lambda(\eta-\sigma)+\Lambda(\sigma)]}\widehat{f}(\xi-\eta,t)
\widehat{f}(\eta-\sigma,t)\widehat{\overline{f}}(\sigma,t)\,d\eta d\sigma,\\
&I_1(\xi,t):=\int_{\mathbb{R}\times\mathbb{R}}e^{it[\Lambda(\xi)-\Lambda(\xi-\eta)-\Lambda(\eta-\sigma)-\Lambda(\sigma)]}\widehat{f}(\xi-\eta,t)
\widehat{f}(\eta-\sigma,t)\widehat{f}(\sigma,t)\,d\eta d\sigma,\\
&I_2(\xi,t):=\int_{\mathbb{R}\times\mathbb{R}}e^{it[\Lambda(\xi)-\Lambda(\xi-\eta)+\Lambda(\eta-\sigma)+\Lambda(\sigma)]}\widehat{f}(\xi-\eta,t)
\widehat{\overline{f}}(\eta-\sigma,t)\widehat{\overline{f}}(\sigma,t)\,d\eta d\sigma,\\
&I_3(\xi,t):=\int_{\mathbb{R}\times\mathbb{R}}e^{it[\Lambda(\xi)+\Lambda(\xi-\eta)+\Lambda(\eta-\sigma)+\Lambda(\sigma)]}\widehat{\overline{f}}(\xi-\eta,t)
\widehat{\overline{f}}(\eta-\sigma,t)\widehat{\overline{f}}(\sigma,t)\,d\eta d\sigma.
\end{split}
\end{equation}
As in Proposition \ref{bootstrap}, for any $t\in[0,T]$ let
\begin{equation}\label{bn2}
H(\xi,t):=\frac{2c_0}{\pi}|\xi|^{3/2}\int_0^t|\widehat{f}(\xi,s)|^2 \frac{ds}{s+1},\qquad g(\xi,t):=e^{iH(\xi,t)}\widehat{f}(\xi,t).
\end{equation}
It follows from \eqref{bn1} that
\begin{equation}\label{bn3}
\begin{split}
(\partial_tg)(\xi,t)=&-ic_0(2\pi)^{-2}e^{iH(\xi,t)}\Big[I_0(\xi,t)-\widetilde{c}\frac{|\xi|^{3/2}|\widehat{f}(\xi,t)|^2}{t+1}\widehat{f}(\xi,t)\Big]\\
&-ie^{iH(\xi,t)}(2\pi)^{-2}[c_1I_1(\xi,t)+c_2I_2(\xi,t)+c_3I_3(\xi,t)],
\end{split}
\end{equation}
where $\widetilde{c}:=8\pi$.

Proposition \ref{bootstrap} clearly follows from Lemma \ref{bigbound1} and Lemma \ref{bigbound2} below.

\begin{lem}\label{bigbound1}
(i) Assume that $f\in C([0,T]:H^{N_0})$ satisfies the identities \eqref{bn1} and $f(0)\in Z\cap W$. Then the mapping $t\to f(t)$ is a continuous mapping from 
$[0,T]$ to $Z\cap W$.

(ii) With $p_0\in(0,1/1000]$, assume, in addition, that
\begin{equation}\label{bn4}
\begin{split}
&\|f(0)\|_{H^{N_0}}+\|f(0)\|_{W}+\|f(0)\|_{Z}=\varepsilon_0\leq 1,\\
&\sup_{t\in[0,T]}\big[(1+t)^{-p_0}\|f(t)\|_{H^{N_0}}+(1+t)^{-p_0}\|f(t)\|_{W}+\|f(t)\|_{Z}\big]\leq\varepsilon_1,
\end{split}
\end{equation}
for some $\varepsilon_1\in[\varepsilon_0,1]$. Then
\begin{equation}\label{bn5}
\sup_{t\in[0,T]}(1+t)^{-p_0}\|f(t)\|_{W}\leq \varepsilon_0+C_{p_0}\varepsilon_1^2.
\end{equation}
\end{lem}

\begin{lem}\label{bigbound2}
With the same notation as before, assume that $f\in C([0,T]:H^{N_0})$ satisfies \eqref{bn1}, and 
\begin{equation}\label{bn7}
\sup_{t\in[0,T]}\big[(1+t)^{-p_0}\|f(t)\|_{H^{N_0}}+(1+t)^{-p_0}\|f(t)\|_{W}+\|f(t)\|_{Z}\big]\leq\varepsilon_1\leq 1.
\end{equation}
Then, for some $p_1>0$,
\begin{equation}\label{bn8}
\sup_{t_1\leq t_2\in[0,T]}(1+t_1)^{p_1}\|(1+|\xi|)^{10}(g(\xi,t_2)-g(\xi,t_1))\|_{L^\infty_\xi}\lesssim\varepsilon_1^3.
\end{equation}
\end{lem}

\subsection{Proof of Lemma \ref{bigbound1}}\label{proof2.1} {\bf{Step 1.}} Assume $t\in\mathbb{R}$, $g\in H^{N_0}\cap Z\cap W$, and let $g^+:=g$, $g^-:=\overline{g}$. We define, for $(\iota_1,\iota_2,\iota_3)\in\{(+,+,-),(+,+,+),(+,-,-),(-,-,-)\}$,
\begin{equation}\label{bn99}
\begin{split}
I_{g,t}^{\iota_1,\iota_2,\iota_3}(\xi)&:=\int_{\mathbb{R}\times\mathbb{R}} e^{it\psi^{\iota_1,\iota_2,\iota_3}(\xi,\eta,\s)}
\widehat{g^{\iota_1}}(\xi-\eta)\widehat{g^{\iota_2}}(\eta-\sigma)\widehat{g^{\iota_3}}(\sigma)\,d\eta d\sigma
\\
\psi^{\iota_1,\iota_2,\iota_3}(\xi,\eta,\s)&:=\Lambda(\xi)-\iota_1\Lambda(\xi-\eta)-\iota_2\Lambda(\eta-\sigma)-\iota_3\Lambda(\sigma).
\end{split}
\end{equation}

It is clear from the definition that
\begin{equation}\label{alo7}
\big\|\mathcal{F}^{-1}(I_{g,t}^{\iota_1,\iota_2,\iota_3})-\mathcal{F}^{-1}(I_{g',t}^{\iota_1,\iota_2,\iota_3})\big\|_{H^{N_0}\cap Z}\lesssim (\|g\|_{H^{N_0}}+\|g'\|_{H^{N_0}})^2\|g-g'\|_{H^{N_0}},
\end{equation}
for any $g,g'\in H^{N_0}\cap Z\cap W$ and $t\in\mathbb{R}$. 

We would like to estimate also $\big\|\mathcal{F}^{-1}(I_{g,t}^{\iota_1,\iota_2,\iota_3})-\mathcal{F}^{-1}(I_{g',t}^{\iota_1,\iota_2,\iota_3})\big\|_{W}$. The key observation is that 
\begin{equation}\label{xidxiphi}
\xi \partial_\xi \psi^{\iota_1,\iota_2,\iota_3}(\xi,\eta,\s) =  - \eta \partial_\eta \psi^{\iota_1,\iota_2,\iota_3}(\xi,\eta,\s) 
- \s \partial_\s \psi^{\iota_1,\iota_2,\iota_3}(\xi,\eta,\s) + \frac{1}{2} \psi^{\iota_1,\iota_2,\iota_3}(\xi,\eta,\s),
\end{equation}
which follows easily from the identity $\xi\Lambda'(\xi)=\Lambda(\xi)/2$ and the definition of $\psi^{\iota_1,\iota_2,\iota_3}$. 
Applying $\xi\partial_\xi$ to $I_{g,t}^{\iota_1,\iota_2,\iota_3}$ we get
\begin{align*}
\xi(\partial I_{g,t}^{\iota_1,\iota_2,\iota_3})(\xi) & = 
\int_{\mathbb{R}\times\mathbb{R}} it (\xi \partial_\xi \psi^{\iota_1,\iota_2,\iota_3})(\xi,\eta,\s) e^{it\psi^{\iota_1,\iota_2,\iota_3}(\xi,\eta,\s)}
\widehat{g^{\iota_1}}(\xi-\eta) \widehat{g^{\iota_2}}(\eta-\sigma)\widehat{g^{\iota_3}}(\sigma)\,d\eta d\sigma
\\
& + \int_{\mathbb{R}\times\mathbb{R}} 
e^{it\psi^{\iota_1,\iota_2,\iota_3}(\xi,\eta,\s)}  \xi(\partial \widehat{g^{\iota_1}})(\xi-\eta) \widehat{g^{\iota_2}}(\eta-\sigma)\widehat{g^{\iota_3}}(\sigma)\,d\eta d\sigma.
\end{align*}
Using \eqref{xidxiphi} to integrate by parts in $\eta$ and $\s$, and gathering terms properly, we see that
\begin{equation}\label{bn100}
\begin{split}
& \xi(\partial I_{g,t}^{\iota_1,\iota_2,\iota_3}) (\xi) = \sum_{j=1}^5 L_{g,t,j}^{\iota_1,\iota_2,\iota_3}(\xi) 
\\
& L_{g,t,1}^{\iota_1,\iota_2,\iota_3}(\xi) := \int_{\mathbb{R}\times\mathbb{R}}
e^{it\psi^{\iota_1,\iota_2,\iota_3}(\xi,\eta,\s)} \cdot(\xi-\eta) \partial \widehat{g^{\iota_1}}(\xi-\eta)\cdot 
\widehat{g^{\iota_2}}(\eta-\sigma)\widehat{g^{\iota_3}}(\sigma)\,d\eta d\sigma
\\
& L_{g,t,2}^{\iota_1,\iota_2,\iota_3}(\xi) := \int_{\mathbb{R}\times\mathbb{R}} 
e^{it\psi^{\iota_1,\iota_2,\iota_3}(\xi,\eta,\s)}  \widehat{g^{\iota_1}}(\xi-\eta)\cdot 
(\eta-\s)\partial \widehat{g^{\iota_2}}(\eta-\sigma)\cdot\widehat{g^{\iota_3}}(\sigma)\,d\eta d\sigma
\\
& L_{g,t,3}^{\iota_1,\iota_2,\iota_3}(\xi) := \int_{\mathbb{R}\times\mathbb{R}} 
e^{it\psi^{\iota_1,\iota_2,\iota_3}(\xi,\eta,\s)} \widehat{g^{\iota_1}}(\xi-\eta) \widehat{g^{\iota_2}}(\eta-\sigma) \cdot\s\partial\widehat{g^{\iota_3}}(\sigma)\,d\eta d\sigma
\\
& L_{g,t,4}^{\iota_1,\iota_2,\iota_3}(\xi) := 2 \int_{\mathbb{R}\times\mathbb{R}} 
e^{it\psi^{\iota_1,\iota_2,\iota_3}(\xi,\eta,\s)} 
\widehat{g^{\iota_1}}(\xi-\eta) \widehat{g^{\iota_2}}(\eta-\sigma) \widehat{g^{\iota_3}}(\sigma)\,d\eta d\sigma
\\
& L_{g,t,5}^{\iota_1,\iota_2,\iota_3}(\xi) := \int_{\mathbb{R}\times\mathbb{R}} \frac{it}{2} \psi^{\iota_1,\iota_2,\iota_3}(\xi,\eta,\s) 
e^{it\psi^{\iota_1,\iota_2,\iota_3}(\xi,\eta,\s)}\widehat{g^{\iota_1}}(\xi-\eta) \widehat{g^{\iota_2}}(\eta-\sigma)\widehat{g^{\iota_3}}(\sigma)\,d\eta d\sigma.
\end{split}
\end{equation}

As a consequence of these formulas it is easy to see that
\begin{equation}\label{alo67}
\begin{split}
\big\|\mathcal{F}^{-1}(I_{g,t}^{\iota_1,\iota_2,\iota_3})-\mathcal{F}^{-1}(I_{g',t}^{\iota_1,\iota_2,\iota_3})\big\|_{W}&\lesssim (1+|t|)(\|g\|_{H^{N_0}}+\|g'\|_{H^{N_0}})^2(\|g-g'\|_{H^{N_0}}+\|g-g'\|_{W})\\
&+(1+|t|)(\|g\|_{H^{N_0}}+\|g'\|_{H^{N_0}})(\|g\|_{W}+\|g'\|_{W})\|g-g'\|_{H^{N_0}},
\end{split}
\end{equation}
for any $g,g'\in H^{N_0}\cap Z\cap W$ and $t\in\mathbb{R}$. In particular, setting $g'=0$,
\begin{equation}\label{alo68}
\big\|\mathcal{F}^{-1}(I_{g,t}^{\iota_1,\iota_2,\iota_3})\big\|_{W}\lesssim (1+|t|)\|g\|_{H^{N_0}}^2(\|g\|_{H^{N_0}}+\|g\|_{W}).
\end{equation}

{\bf{Step 2.}} We can prove now part (i) of the lemma, using a standard fixed-point argument. Indeed, given an interval $I\subseteq\mathbb{R}$, a point $t_0\in I$, and a function $g\in C(I:H^{N_0})$, we define
\begin{equation*}
\widehat{\Gamma(g)}(\xi,t):=\widehat{g_0}(\xi) - i (2\pi)^{-2}\int_{t_0}^t c_0I_{g(s),s}^{+,+,-}(\xi)+c_1I_{g(s),s}^{+,+,+}(\xi)+c_2I_{g(s),s}^{+,-,-}(\xi)+c_3I_{g(s),s}^{-,-,-}(\xi)\,ds,
\end{equation*}
where $g_0\in H^{N_0}\cap Z\cap W$. It follows from \eqref{alo7} and \eqref{alo67} that the mapping $g\to\Gamma(g)$ is a contraction on the complete metric space
\begin{equation*}
\begin{split}
&\mathcal{M}:=\{h\in C(I:H^{N_0}\cap Z\cap W):\sup_{t\in I}\|h(t)\|_{H^{N_0}\cap W}\leq 2\|g_0\|_{H^{N_0}\cap W}\},\\
&d_\mathcal{M}(h,h'):=\sup_{t\in I}\|h-h'\|_{H^{N_0}\cap Z\cap W},
\end{split}
\end{equation*} 
provided that $|I|$ is sufficiently small (depending only on $t_0$ and $\|g_0\|_{H^{N_0}\cap W}$). 

With the notation in the statement of the lemma, we notice that if $T'\leq T$ and $f\in C([0,T']:H^{N_0}\cap W)$ then
\begin{equation}\label{alo71}
\sup_{t\in[0,T']}\|f(t)\|_W\leq C(T,\sup_{t\in T}\|f(t)\|_{H^{N_0}},\|f(0)\|_{W}).
\end{equation}
Indeed, the bound \eqref{alo71} follows from the identity
\begin{equation}\label{alo72}
\widehat{f}(t_2)-\widehat{f}(t_1)=- i (2\pi)^{-2}\int_{t_1}^{t_2} c_0I_{f(s),s}^{+,+,-}(\xi)+c_1I_{f(s),s}^{+,+,+}(\xi)+c_2I_{f(s),s}^{+,-,-}(\xi)+c_3I_{f(s),s}^{-,-,-}(\xi)\,ds,
\end{equation}
and the bound \eqref{alo68}. 

Therefore we can divide the interval $[0,T]$ into finitely many subintervals, with sufficiently small length depending only on $T$, $\sup_{t\in T}\|f(t)\|_{H^{N_0}}$, and $\|f(0)\|_{W}$. We apply then the fixed-point argument above on each such subinterval, which is possible in view of the uniform bound \eqref{alo71}. It follows that $f\in C([0,T]:H^{N_0}\cap Z\cap W)$, and
\begin{equation*}
\sup_{t\in[0,T]}\|f\|_{H^{N_0}\cap Z\cap W}\leq C(T,\sup_{t\in T}\|f(t)\|_{H^{N_0}},\|f(0)\|_{H^{N_0}\cap Z\cap W}),
\end{equation*}
as desired.

{\bf{Step 3.}} To prove part (ii) we need to improve on the uniform {\it{apriori}} bound \eqref{alo71}, provided that the solution $f$ satisfies the stronger assumptions \eqref{bn4}. We use the formula \eqref{alo72} and reexamine the decomposition \eqref{bn100}. It follows that, for any $t_1,t_2\in[0,T]$,
\begin{equation}\label{alo75}
\begin{split}
\|f(t_2)-f(t_1)\|_W&\lesssim\sum_{j=1}^5\sum_{(\iota_1,\iota_2,\iota_3)}\Big\|\int_{t_1}^{t_2}L_{f(s),s,j}^{\iota_1,\iota_2,\iota_3}(\xi)\,ds\Big\|_{L^2_\xi}\\
&\lesssim \int_{t_1}^{t_2}\|f(s)\|_{W}\|e^{-is\Lambda} f(s)\|_{L^\infty}^2\,ds+ \sum_{(\iota_1,\iota_2,\iota_3)}\Big\|\int_{t_1}^{t_2}L_{f(s),s,5}^{\iota_1,\iota_2,\iota_3}(\xi)\,ds\Big\|_{L^2_\xi}.
\end{split}
\end{equation}

To estimate the contribution coming from $L_{f(s),s,5}^{\iota_1,\iota_2,\iota_3}$, we integrate by parts in $s$ using the identity
$$ \psi^{\iota_1,\iota_2,\iota_3}(\xi,\eta,\s)  e^{is\psi^{\iota_1,\iota_2,\iota_3}(\xi,\eta,\s)}
= -i \partial_s e^{is\psi^{\iota_1,\iota_2,\iota_3}(\xi,\eta,\s)}.
$$
We obtain
\begin{equation}
\begin{split}
\label{bn110}
\Big\|\int_{t_1}^{t_2} L_{f(s),s,5}^{\iota_1,\iota_2,\iota_3}&(\xi) ds \Big\|_{L^2_\xi}\lesssim \sum_{j=1}^2|t_j|\Big\|\int_{\mathbb{R}\times\mathbb{R}} e^{it_j\psi^{\iota_1,\iota_2,\iota_3}(\xi,\eta,\s)}\widehat{f^{\iota_1}}(\xi-\eta,t_j) \widehat{f^{\iota_t}}(\eta-\sigma,t_j)\widehat{f^{\iota_t}}(\sigma,t_j)\,d\eta d\sigma\Big\|_{L^2_\xi}
\\
& + \int_{t_1}^{t_2}s\Big\| \int_{\mathbb{R}\times\mathbb{R}}e^{is\psi^{\iota_1,\iota_2,\iota_3}(\xi,\eta,\s)} \partial_s \left[ \widehat{f^{\iota_1}}(\xi-\eta,s) \widehat{f^{\iota_2}}(\eta-\sigma,s)\widehat{f^{\iota_3}}(\sigma,s) \right] \,d\eta d\sigma \Big\|_{L^2_\xi}\, ds
\\
& + \int_{t_1}^{t_2} \Big\| \int_{\mathbb{R}\times\mathbb{R}} e^{is\psi^{\iota_1,\iota_2,\iota_3}(\xi,\eta,\s)} \left[ \widehat{f^{\iota_1}}(\xi-\eta,s) \widehat{f^{\iota_2}}(\eta-\sigma,s)\widehat{f^{\iota_3}}(\sigma,s) \right] \,d\eta d\sigma \Big\|_{L^2_\xi}\, ds.
\end{split}
\end{equation}

The term in the last line of \eqref{bn110} is majorized by
\begin{equation*}
C\int_{t_1}^{t_2}\|f(s)\|_{W}\|e^{-is\Lambda} f(s)\|_{L^\infty}^2\,ds.
\end{equation*}
The terms in the first line of \eqref{bn110} are majorized by
\begin{equation*}
C\sum_{j=1}^2|t_j|\| f(t_j) \|_{L^2}  \|e^{-it_j\Lambda} f(t_j)\|_{L^\infty}^2.
\end{equation*}
Finally, using also the identities \eqref{bn1}, the term in the second line of \eqref{bn110} is majorized by
\begin{equation*}
C\int_{t_1}^{t_2}s\|e^{-is\Lambda} f(s)\|_{L^\infty}^2\|\partial_s\widehat{f}(s)\|_{L^2}\,ds\lesssim \int_{t_1}^{t_2}s\|e^{-is\Lambda} f(s)\|_{L^\infty}^4\|\widehat{f}(s)\|_{L^2}\,ds.
\end{equation*}

Using the assumption \eqref{bn4} and Lemma \ref{dispersive}, we have
\begin{equation*}
\|e^{-is\Lambda} f(s)\|_{L^\infty}\lesssim \varepsilon_1(1+|s|)^{-1/2},\qquad \|f(s)\|_{H^{N_0}}+\|f(s)\|_W\lesssim \varepsilon_1(1+|s|)^{p_0}.
\end{equation*}
Therefore, using also \eqref{alo75}, 
\begin{equation*}
\|f(t)-f(0)\|_W\lesssim_{p_0}\varepsilon_1^3(1+t)^{p_0},
\end{equation*}
for any $t\in[0,T]$. The desired estimate \eqref{bn5} follows, which completes the proof of the lemma.

\section{Proof of Lemma \ref{bigbound2}}\label{proof2.2} In this section we give the proof of Lemma \ref{bigbound2}, which is the more technical part of the paper. 
With $P_k$ defined as in section \ref{proof1}, we let $f_k^+:=P_kf$, $f_k^-:=P_k\overline{f}$, and decompose
\begin{equation}\label{bn9}
\begin{split}
&I_0=\sum_{k_1,k_2,k_3\in\mathbb{Z}}I_{k_1,k_2,k_3}^{+,+,-},\\
&I_1=\sum_{k_1,k_2,k_3\in\mathbb{Z}}I_{k_1,k_2,k_3}^{+,+,+},
\quad I_2=\sum_{k_1,k_2,k_3\in\mathbb{Z}}I_{k_1,k_2,k_3}^{+,-,-},\quad I_3=\sum_{k_1,k_2,k_3\in\mathbb{Z}}I_{k_1,k_2,k_3}^{-,-,-},
\end{split}
\end{equation}
where, for $(\iota_1,\iota_2,\iota_3)\in\{(+,+,-),(+,+,+),(+,-,-),(-,-,-)\}$,
\begin{equation}\label{bn10} I_{k_1,k_2,k_3}^{\iota_1,\iota_2,\iota_3}(\xi,t):=\int_{\mathbb{R}\times\mathbb{R}}e^{it[\Lambda(\xi)-\iota_1\Lambda(\xi-\eta)-\iota_2\Lambda(\eta-\sigma)-\iota_3\Lambda(\sigma)]}
\widehat{f_{k_1}^{\iota_1}}(\xi-\eta,t)\widehat{f_{k_2}^{\iota_2}}(\eta-\sigma,t)\widehat{f_{k_3}^{\iota_3}}(\sigma,t)\,d\eta d\sigma.
\end{equation}

For \eqref{bn8} it suffices to prove that if $t_1\leq t_2\in[2^m-2,2^{m+1}]\cap[0,T]$, for some $m\in\{1,2\ldots\}$, then
\begin{equation*}
\|(1+|\xi|)^{10}(g(\xi,t_2)-g(\xi,t_1))\|_{L^\infty_\xi}\lesssim\varepsilon_1^32^{-p_1m}.
\end{equation*}
Using \eqref{bn3} and the decompositions \eqref{bn9}, it suffices to prove that if $k\in\mathbb{Z}$, $m\in\{1,2,\ldots\}$, $|\xi|\in[2^k,2^{k+1}]$, 
and $t_1\leq t_2\in[2^m-2,2^{m+1}]\cap[0,T]$ then
\begin{equation}\label{bn11}
\begin{split}
\sum_{k_1,k_2,k_3\in\mathbb{Z}}\Big|\int_{t_1}^{t_2} e^{iH(\xi,s)}\Big[I_{k_1,k_2,k_3}^{+,+,-}(\xi,s)-\widetilde{c}\frac{|\xi|^{3/2}
\widehat{f_{k_1}^+}(\xi,s)\widehat{f_{k_2}^+}(\xi,s)\widehat{f_{k_3}^-}(-\xi,s)}{s+1}\Big]\,ds\Big|\\
\lesssim \varepsilon_1^32^{-p_1m}2^{-10k_+},
\end{split}
\end{equation}
and, for any $(\iota_1,\iota_2,\iota_3)\in\{(+,+,+),(+,-,-),(-,-,-)\}$,
\begin{equation}\label{bn12}
\sum_{k_1,k_2,k_3\in\mathbb{Z}}\Big|\int_{t_1}^{t_2} e^{iH(\xi,s)}I_{k_1,k_2,k_3}^{\iota_1,\iota_2,\iota_3}(\xi,s)\,ds\Big|
\lesssim \varepsilon_1^32^{-p_1m}2^{-10k_+}.
\end{equation}

In view of \eqref{bn7}, we have
\begin{equation}\label{bn13}
\begin{split}
\|\widehat{f_l^{\pm}}(s)\|_{L^2}&\lesssim \varepsilon_12^{p_0 m}2^{-N_0 l_+},\\
\|(\partial\widehat{f_l^{\pm}})(s)\|_{L^2}&\lesssim \varepsilon_12^{p_0 m}2^{-l},\\
\|\widehat{f_l^{\pm}}(s)\|_{L^\infty}&\lesssim \varepsilon_12^{-10l_+},
\end{split}
\end{equation}
for any $l\in\mathbb{Z}$ and $s\in[2^m-2,2^{m+1}]\cap[0,T]$. Using only the $L^2$ bounds in the first line of \eqref{bn13} it is easy to see that
\begin{equation}\label{bn14}
|I_{k_1,k_2,k_3}^{\iota_1,\iota_2,\iota_3}(\xi,s)|\lesssim \varepsilon_1^32^{3p_0 m}2^{\min(k_1,k_2,k_3)/2}(1+2^{\max(k_1,k_2,k_3)})^{-N_0},
\end{equation}
for any $(\iota_1,\iota_2,\iota_3)\in\{(+,+,-),(+,+,+),(+,-,-),(-,-,-)\}$, $k_1,k_2,k_3\in\mathbb{Z}$. Moreover , using the $L^\infty$ bounds in the last 
line of \eqref{bn13},
\begin{equation*}
\Big|\frac{|\xi|^{3/2}\widehat{f_{k_1}^+}(\xi,s)\widehat{f_{k_2}^+}(\xi,s)\widehat{f_{k_3}^-}(\xi,s)}{s+1}\Big|
\lesssim 2^{-m}\varepsilon_1^32^{3k/2}(2^{\alpha k}+2^{10k})^{-3}\mathbf{1}_{[0,4]}(\max(|k_1-k|,|k_2-k|,|k_3-k|)).
\end{equation*}
Using these two bounds it is easy to see that the sums in \eqref{bn11} and \eqref{bn12} over those $(k_1,k_2,k_3)$ for which 
$\max(k_1,k_2,k_3)\geq m/50-1000$ or $\min(k_1,k_2,k_3)\leq -4m$ are bounded by $C\varepsilon_1^32^{-p_1m}2^{-10k_+}$, as desired. 
The remaining sums have only $Cm^3$ terms. Therefore it suffices to prove the desired estimates for each $(k_1,k_2,k_3)$ fixed; more precisely it suffices to prove the 
following lemma:

\begin{lem}\label{bb1}
Assume that $k\in\mathbb{Z}$, $m\in\mathbb{Z}\cap[20,\infty)$, $|\xi|\in[2^k,2^{k+1}]$, $t_1\leq t_2\in[2^{m-1},2^{m+1}]\cap[0,T]$, 
and $k_1,k_2,k_3\in[-4m,m/50-1000]\cap\mathbb{Z}$. Then
\begin{equation}\label{bn20}
\Big|\int_{t_1}^{t_2} e^{iH(\xi,s)}\Big[I_{k_1,k_2,k_3}^{+,+,-}(\xi,s)-\widetilde{c}\frac{|\xi|^{3/2}\widehat{f_{k_1}^+}(\xi,s)\widehat{f_{k_2}^+}(\xi,s)
\widehat{f_{k_3}^-}(-\xi,s)}{s+1}\Big]\,ds\Big|\lesssim \varepsilon_1^32^{-2p_1m}2^{-10k_+},
\end{equation}
and, for any $(\iota_1,\iota_2,\iota_3)\in\{(+,+,+),(+,-,-),(-,-,-)\}$,
\begin{equation}\label{bn21}
\Big|\int_{t_1}^{t_2} e^{iH(\xi,s)}I_{k_1,k_2,k_3}^{\iota_1,\iota_2,\iota_3}(\xi,s)\,ds\Big|\lesssim \varepsilon_1^32^{-2p_1m}2^{-10k_+}.
\end{equation}
\end{lem}

We will prove this main lemma in several steps. More precisely, the bounds \eqref{bn20} follow from Lemma \ref{bb2}, Lemma \ref{bb10}, Lemma \ref{bb11}, Lemma \ref{bb12}, and Lemma \ref{bb13}. The bounds \eqref{bn21} follow from Lemma \ref{bb10} and Lemma \ref{bb14}.

We will use the bounds \eqref{bn13} and the $L^\infty$ bounds
\begin{equation}\label{bn13.5}
 \|e^{\mp is\Lambda}f_l^{\pm}(s)\|_{L^\infty}\lesssim \varepsilon_12^{-m/2},\qquad\text{ for any }l\in\mathbb{Z}\text{ and }s\in[t_1,t_2],
\end{equation}
which follow from \eqref{bn13} and Lemma \ref{dispersive}. We will also use the bounds in Lemma \ref{touse} and Lemma \ref{touse3} below:

\begin{lem}\label{touse}
 Assume that $m\in L^1(\mathbb{R}\times\mathbb{R})$ satisfies
\begin{equation}\label{touse1}
\Big\|\,\int_{\mathbb{R}\times\mathbb{R}}m(\eta,\sigma)e^{ix\eta}e^{iy\sigma}\,d\eta d\sigma\Big\|_{L^1_{x,y}}\leq A,
\end{equation}
for some $A\in(0,\infty)$. Then, for any $(p,q,r)\in\{(2,2,\infty),(2,\infty,2),(\infty,2,2)\}$,
\begin{equation}\label{touse2}
 \Big|\int_{\mathbb{R}\times\mathbb{R}}\widehat{f}(\eta)\widehat{g}(\sigma)\widehat{h}(-\eta-\sigma)m(\eta,\sigma)\,d\eta d\sigma\Big|
\lesssim A\|f\|_{L^p}\|g\|_{L^q}\|h\|_{L^r}.
\end{equation}
\end{lem}

\begin{proof}[Proof of Lemma \ref{touse}] We rewrite
\begin{equation*}
\begin{split}
\Big|\int_{\mathbb{R}\times\mathbb{R}}\widehat{f}(\eta)\widehat{g}(\sigma)\widehat{h}(-\eta-\sigma)m(\eta,\sigma)\,d\eta d\sigma\Big|
&=C\Big|\int_{\mathbb{R}^3}f(x)g(y)h(z)K(z-x,z-y)\,dxdydz\Big|,\\
&\lesssim \int_{\mathbb{R}^3}|f(z-x)g(z-y)h(z)|\,|K(x,y)|\,dxdydz,
\end{split}
\end{equation*}
where
\begin{equation*}
 K(x,y):=\int_{\mathbb{R}\times\mathbb{R}}m(\eta,\sigma)e^{ix\eta}e^{iy\sigma}\,d\eta d\sigma.
\end{equation*}
The desired bound \eqref{touse2} follows easily from \eqref{touse1}.
\end{proof}

\begin{lem}\label{touse3}
For any $l\in\mathbb{Z}$ and $s\in[2^{m-1},2^{m+1}]\cap[0,T]$ we have
\begin{equation}\label{touse4}
\|(\partial_s\widehat{f^{\pm}_l})(s)\|_{L^2}\lesssim \varepsilon_12^{3p_0m}2^{-20l_+}2^{-m}
\end{equation}
and
\begin{equation}\label{touse4.1}
\|(\partial_s\widehat{f^{\pm}_l})(s)\|_{L^\infty}\lesssim \varepsilon_12^{3p_0m}2^{-20l_+}2^{-m/2}(2^{l/2}+2^{-m/2}).
\end{equation}
\end{lem}

\begin{proof}[Proof of Lemma \ref{touse3}] Using the identity \eqref{bn1}, it suffices to prove that
\begin{equation*}
\begin{split}
&\|\varphi_l(\xi)I_d(\xi,s)\|_{L^2_\xi}\lesssim \varepsilon_12^{3p_0m}2^{-20l_+}2^{-m},\\
&\|\varphi_l(\xi)I_d(\xi,s)\|_{L^\infty_\xi}\lesssim \varepsilon_12^{3p_0m}2^{-20l_+}2^{-m/2}(2^{l/2}+2^{-m/2}),
\end{split}
\end{equation*}
for $d\in\{0,1,2,3\}$. Using the decomposition \eqref{bn9}--\eqref{bn10}, it suffices to prove that for any $(\iota_1,\iota_2,\iota_3)\in\{(+,+,-),(+,+,+),(+,-,-),(-,-,-)\}$,
\begin{equation}\label{touse4.2}
\begin{split}
\sum_{\max(k_1,k_2,k_3)\geq l-10}\Big\|&\int_{\mathbb{R}\times\mathbb{R}}e^{is[-\iota_1\Lambda(\xi-\eta)-\iota_2\Lambda(\eta-\sigma)-\iota_3\Lambda(\sigma)]}\\
&\times\widehat{f_{k_1}^{\iota_1}}(\xi-\eta,s)\widehat{f_{k_2}^{\iota_2}}(\eta-\sigma,s)\widehat{f_{k_3}^{\iota_3}}(\sigma,s)\,d\eta d\sigma\Big\|_{L^2_\xi}\lesssim \varepsilon_12^{3p_0m}2^{-20l_+}2^{-m}
\end{split}
\end{equation}
and
\begin{equation}\label{touse4.3}
\begin{split}
\sum_{\max(k_1,k_2,k_3)\geq l-10}\Big\|\varphi_l(\xi)&\int_{\mathbb{R}\times\mathbb{R}}e^{is[-\iota_1\Lambda(\xi-\eta)-\iota_2\Lambda(\eta-\sigma)-\iota_3\Lambda(\sigma)]}\widehat{f_{k_1}^{\iota_1}}(\xi-\eta,s)\\
&\times\widehat{f_{k_2}^{\iota_2}}(\eta-\sigma,s)\widehat{f_{k_3}^{\iota_3}}(\sigma,s)\,d\eta d\sigma\Big\|_{L^\infty_\xi}\lesssim \varepsilon_12^{3p_0m}2^{-20l_+}2^{-m/2}(2^{l/2}+2^{-m/2}).
\end{split}
\end{equation}

We use first the bounds
\begin{equation}\label{touse4.4}
\|e^{\mp is\Lambda}f^{\pm}_n(s)\|_{L^2}\lesssim \varepsilon_12^{p_0m}2^{-N_0n_+}2^{n/2},\qquad \|e^{\mp is\Lambda}f^{\pm}_n(s)\|_{L^\infty}\lesssim \varepsilon_1\min(2^{-m/2},2^n2^{-10n_+}),
\end{equation}
see \eqref{bn13} and \eqref{bn13.5}. The bound \eqref{touse4.3} follows by passing to the physical space and estimating the highest frequency component in $L^2$ and the other two components in $L^\infty$. The bound \eqref{touse4.2} also follows if $l\geq 0$, by passing to the physical space and estimating the two highest frequency components in $L^2$ and the lowest frequency component in $L^\infty$. 

On the other hand, if $l\leq 0$ then we can still use \eqref{touse4.4} to estimate the contribution of the sum over $\min(k_1,k_2,k_3)\leq \max(l,-m)+10$ in \eqref{touse4.3}, or the contribution of the sum over $\max(k_1,k_2,k_3)\geq m/10$. Therefore, for \eqref{touse4.3} it remains to prove that
\begin{equation}\label{touse4.5}
\begin{split}
\Big|\int_{\mathbb{R}\times\mathbb{R}}e^{is[-\iota_1\Lambda(\xi-\eta)-\iota_2\Lambda(\eta-\sigma)-\iota_3\Lambda(\sigma)]}\widehat{f_{k_1}^{\iota_1}}(\xi-\eta,s)\widehat{f_{k_2}^{\iota_2}}(\eta-\sigma,s)\widehat{f_{k_3}^{\iota_3}}(\sigma,s)\,d\eta d\sigma\Big|\\
\lesssim \varepsilon_12^{2p_0m}2^{-m/2}(2^{l/2}+2^{-m/2}),
\end{split}
\end{equation}
provided that
\begin{equation}\label{touse4.6}
|\xi|\in[2^{l-1},2^{l+1}],\qquad l\leq 0,\qquad \max(l,-m)+10\leq k_1,k_2,k_3\leq m/10.
\end{equation}

In proving \eqref{touse4.5} we may assume, without loss of generality, that $k_1=\min(k_1,k_2,k_3)$ and therefore $|k_2-k_3|\leq 4$. We decompose the integral in the left-hand side of \eqref{touse4.5} into two parts, depending on the relative sizes of $|\eta|$ and $|\sigma|$, and integrate by parts. More precisely, let $\chi:\mathbb{R}\to[0,1]$ denote an even smooth function supported in $[-11/10,11/10]$ and equal to $1$ in $[-9/10,9/10]$, and define
\begin{equation*}
\begin{split}
&J_1=\int_{\mathbb{R}\times\mathbb{R}}\chi(\eta/\sigma)e^{is[-\iota_1\Lambda(\xi-\eta)-\iota_2\Lambda(\eta-\sigma)-\iota_3\Lambda(\sigma)]}\widehat{f_{k_1}^{\iota_1}}(\xi-\eta,s)\widehat{f_{k_2}^{\iota_2}}(\eta-\sigma,s)\widehat{f_{k_3}^{\iota_3}}(\sigma,s)\,d\eta d\sigma,\\
&J_2=\int_{\mathbb{R}\times\mathbb{R}}(1-\chi)(\eta/\sigma)e^{is[-\iota_1\Lambda(\xi-\eta)-\iota_2\Lambda(\eta-\sigma)-\iota_3\Lambda(\sigma)]}\widehat{f_{k_1}^{\iota_1}}(\xi-\eta,s)\widehat{f_{k_2}^{\iota_2}}(\eta-\sigma,s)\widehat{f_{k_3}^{\iota_3}}(\sigma,s)\,d\eta d\sigma.
\end{split}
\end{equation*}
To estimate $|J_1|$ we integrate by parts in $\sigma$. Recall that $\Lambda(\theta)=|\theta|^{1/2}$, which shows that
\begin{equation*}
|\partial_\sigma[-\iota_1\Lambda(\xi-\eta)-\iota_2\Lambda(\eta-\sigma)-\iota_3\Lambda(\sigma)]|\gtrsim \big||\Lambda'(\sigma-\eta)|-|\Lambda'(\sigma)|\big|\gtrsim 2^{k_1-3k_2/2},
\end{equation*}
provided that \eqref{touse4.6} holds, and, in addition, $|\eta/\sigma|\leq 11/10$, $|\xi-\eta|\in[2^{k_1-2},2^{k_1+2}]$, $|\eta-\sigma|\in[2^{k_2-2},2^{k_2+2}]$, and $|\sigma|\in[2^{k_3-2},2^{k_3+2}]$. Therefore, after integration by parts in $\sigma$, we estimate
\begin{equation*}
\begin{split}
|J_1|\lesssim \int_{\mathbb{R}\times\mathbb{R}}\frac{1}{2^m2^{k_1-3k_2/2}}&|\widehat{f_{k_1}^{\iota_1}}(\xi-\eta,s)|\Big[2^{-k_2}|\widehat{f_{k_2}^{\iota_2}}(\eta-\sigma,s)\widehat{f_{k_3}^{\iota_3}}(\sigma,s)|\\
&+|(\partial\widehat{f_{k_2}^{\iota_2}})(\eta-\sigma,s)||\widehat{f_{k_3}^{\iota_3}}(\sigma,s)|+|\widehat{f_{k_2}^{\iota_2}}(\eta-\sigma,s)||(\partial\widehat{f_{k_3}^{\iota_3}})(\sigma,s)|\Big]\,d\eta d\sigma.
\end{split}
\end{equation*}
Using the bounds in \eqref{bn13} it follows that
\begin{equation}\label{touse4.8}
|J_1|\lesssim 2^{2p_0m}2^{-m}.
\end{equation}

We estimate now $|J_2|$. Recalling the assumption $k_1=\min(k_1,k_2,k_3)$, we observe that $J_2$ vanishes unless $k_2,k_3\in[k_1,k_1+4]$. In this case we notice that \begin{equation*}
|\partial_\eta[-\iota_1\Lambda(\xi-\eta)-\iota_2\Lambda(\eta-\sigma)-\iota_3\Lambda(\sigma)]|\gtrsim \big||\Lambda'(\eta-\xi)|-|\Lambda'(\eta-\sigma)|\big|\gtrsim 2^{-k_2/2},
\end{equation*}
provided that \eqref{touse4.6} holds, and, in addition, $|\eta/\sigma|\geq 9/10$, $|\xi-\eta|\in[2^{k_1-2},2^{k_1+2}]$, $|\eta-\sigma|\in[2^{k_2-2},2^{k_2+2}]$, and $|\sigma|\in[2^{k_3-2},2^{k_3+2}]$. Therefore, after integration by parts in $\eta$, we estimate
\begin{equation*}
\begin{split}
|J_2|\lesssim \int_{\mathbb{R}\times\mathbb{R}}\frac{1}{2^m2^{-k_2/2}}&|\widehat{f_{k_3}^{\iota_3}}(\sigma,s)|\Big[2^{-k_2}|\widehat{f_{k_1}^{\iota_1}}(\xi-\eta,s)||\widehat{f_{k_2}^{\iota_2}}(\eta-\sigma,s)|\\
&+|(\partial\widehat{f_{k_1}^{\iota_1}})(\xi-\eta,s)||\widehat{f_{k_2}^{\iota_2}}(\eta-\sigma,s)|+|\widehat{f_{k_1}^{\iota_1}}(\xi-\eta,s)||(\partial\widehat{f_{k_2}^{\iota_2}})(\eta-\sigma,s)|\Big]\,d\eta d\sigma.
\end{split}
\end{equation*}
Using the bounds in \eqref{bn13} we see that $|J_2|\lesssim 2^{2p_0m}2^{-m}$ as well. The desired bound \eqref{touse4.5} follows using also \eqref{touse4.8}, which completes the proof of the lemma.
\end{proof}

\subsection{Proof of Lemma \ref{bb1}} We divide the proof into several parts. 

\begin{lem}\label{bb2}
The bounds \eqref{bn20} hold provided that
\begin{equation}\label{bn25}
k_1,k_2,k_3\in[k-20,k+20]\cap\mathbb{Z}.
\end{equation}
\end{lem}

\begin{proof}[Proof of Lemma \ref{bb2}]
 This is the main case, when the specific correction in the left-hand side of \eqref{bn20} is important. We will prove that
\begin{equation}\label{bn28}
 \Big|I_{k_1,k_2,k_3}^{+,+,-}(\xi,s)-\widetilde{c}\frac{|\xi|^{3/2}\widehat{f_{k_1}^+}(\xi,s)\widehat{f_{k_2}^+}(\xi,s)
\widehat{f_{k_3}^-}(-\xi,s)}{s+1}\Big|\lesssim 2^{-m}\varepsilon_1^32^{-2p_1m}2^{-10k_+},
\end{equation}
for any $s\in [t_1,t_2]$, which is clearly stronger than the desired bound \eqref{bn20}.

The bound \eqref{bn28} follows easily from the bound in the last line of \eqref{bn13} if $k\leq-3m/5$. Therefore, in the rest of the proof 
of \eqref{bn28} we may assume that
\begin{equation}\label{bn28.5}
 k\geq -3m/5.
\end{equation}

After changes of variables we rewrite
\footnote{The point of this change of variables is to be able to identify $\eta=\sigma=0$ as the unique critical point of the phase $\Phi$ in \eqref{bn26}.} 
\begin{equation*}
I_{k_1,k_2,k_3}^{+,+,-}(\xi,s)=\int_{\mathbb{R}^2}e^{is\Phi(\xi,\eta,\sigma)}
\widehat{f_{k_1}^+}(\xi+\eta,s)\widehat{f_{k_2}^+}(\xi+\sigma,s)\widehat{f_{k_3}^-}(-\xi-\eta-\sigma,s)\,d\eta d\sigma,
\end{equation*}
where
\begin{equation}\label{bn26}
\Phi(\xi,\eta,\sigma):=\Lambda(\xi)-\Lambda(\xi+\eta)-\Lambda(\xi+\sigma)+\Lambda(\xi+\eta+\sigma).
\end{equation}
Let $\overline{l}$ denote the smallest integer with the property that $2^{\overline{l}}\geq 2^{3k/4}2^{-49m/100}$ (in view of \eqref{bn28.5} $\overline{l}\leq k-10$),  and decompose
\begin{equation}\label{bn27}
I_{k_1,k_2,k_3}^{+,+,-}(\xi,s)=\sum_{l_1,l_2=\overline{l}}^{k+20}J_{l_1,l_2}(\xi,s),
\end{equation}
where, with the notation in \eqref{disp0}, for any $l_1,l_2\geq \overline{l}$,
\begin{equation}\label{bn29}
J_{l_1,l_2}(\xi,s):=\int_{\mathbb{R}^2}e^{is\Phi(\xi,\eta,\sigma)}\widehat{f_{k_1}^+}(\xi+\eta,s)\widehat{f_{k_2}^+}(\xi+\sigma,s)
\widehat{f_{k_3}^-}(-\xi-\eta-\sigma,s)\varphi^{(\overline{l})}_{l_1}(\eta)\varphi^{(\overline{l})}_{l_2}(\sigma)\,d\eta d\sigma.
\end{equation}

{\bf{Step 1.}} We show first that 
\begin{equation}\label{bn30}
|J_{l_1,l_2}(\xi,s)|\lesssim 2^{-m}\varepsilon_1^32^{-3p_1m}2^{-10k_+},\qquad\text{ if }l_2\geq\max(l_1,\overline{l}+1).
\end{equation}
For this we integrate by parts in $\eta$ in the formula \eqref{bn29}. Recalling that $\Lambda(\theta)=\sqrt{|\theta|}$, we observe that
\begin{equation}\label{bn30.5}
\big|(\partial_\eta\Phi)(\xi,\eta,\sigma)\big|=\big|\Lambda'(\xi+\eta+\sigma)-\Lambda'(\xi+\eta)\big|\gtrsim 2^{l_2}2^{-3k/2},
\end{equation}
provided that $|\xi+\eta|\approx 2^k, |\xi+\eta+\sigma|\approx 2^k, |\sigma|\approx 2^{l_2}$. After integration by parts in $\eta$ we see that
\begin{equation*}
|J_{l_1,l_2}(\xi,s)|\leq |J_{l_1,l_2,1}(\xi,s)|+|F_{l_1,l_2,1}(\xi,s)|+|G_{l_1,l_2,1}(\xi,s)|,
\end{equation*}
where
\begin{equation}\label{bn31}
\begin{split}
&J_{l_1,l_2,1}(\xi,s):=\int_{\mathbb{R}^2}e^{is\Phi(\xi,\eta,\sigma)}\widehat{f_{k_1}^+}(\xi+\eta,s)\widehat{f_{k_2}^+}(\xi+\sigma,s)
\widehat{f_{k_3}^-}(-\xi-\eta-\sigma,s)(\partial_\eta m_1)(\eta,\sigma)\,d\eta d\sigma,\\
&F_{l_1,l_2,1}(\xi,s):=\int_{\mathbb{R}^2}e^{is\Phi(\xi,\eta,\sigma)}(\partial\widehat{f_{k_1}^+})(\xi+\eta,s)\widehat{f_{k_2}^+}(\xi+\sigma,s)
\widehat{f_{k_3}^-}(-\xi-\eta-\sigma,s)m_1(\eta,\sigma)\,d\eta d\sigma,\\
&G_{l_1,l_2,1}(\xi,s):=\int_{\mathbb{R}^2}e^{is\Phi(\xi,\eta,\sigma)}\widehat{f_{k_1}^+}(\xi+\eta,s)\widehat{f_{k_2}^+}(\xi+\sigma,s)
(\partial\widehat{f_{k_3}^-})(-\xi-\eta-\sigma,s)m_1(\eta,\sigma)\,d\eta d\sigma,
\end{split}
\end{equation}
and
\begin{equation*}
 m_1(\eta,\sigma):=\frac{\varphi^{(\overline{l})}_{l_1}(\eta)\varphi_{l_2}(\sigma)}{s(\partial_\eta\Phi)(\xi,\eta,\sigma)}\cdot 
\varphi_{[k_1-2,k_1+2]}(\xi+\eta)\varphi_{[k_3-2,k_3+2]}(\xi+\eta+\sigma).
\end{equation*}

To estimate $|F_{l_1,l_2,1}(\xi,s)|$ we recall that $\xi$ and $s$ are fixed and use Lemma \ref{touse} with
\begin{equation*}
\begin{split}
&\widehat{f}(\eta):=e^{-is\Lambda(\xi+\eta)}(\partial\widehat{f_{k_1}^+})(\xi+\eta,s),\\
&\widehat{g}(\sigma):=e^{-is\Lambda(\xi+\sigma)}\widehat{f_{k_2}^+}(\xi+\sigma,s)\cdot \varphi(\sigma/2^{l_2+4}),\\
&\widehat{h}(\theta):=e^{is\Lambda(\xi-\theta)}\widehat{f_{k_3}^-}(-\xi+\theta,s)\cdot \varphi(\theta/2^{l_2+4}).
\end{split}
\end{equation*}
It is easy to see, compare with \eqref{bn30.5}, that $m_1$ satisfies the symbol-type estimates
\begin{equation}\label{bn32}
|(\partial_\eta^a\partial_\sigma^bm_1)(\eta,\sigma)|\lesssim (2^{-m}2^{-l_2}2^{3k/2})(2^{-al_1}2^{-bl_2})\cdot 
\mathbf{1}_{[0,2^{l_1+4}]}(|\eta|)\mathbf{1}_{[2^{l_2-4},2^{l_2+4}]}(|\sigma|),
\end{equation}
for any $a,b\in[0,20]\cap\mathbb{Z}$. It follows from \eqref{bn13} and \eqref{bn13.5} that
\begin{equation*}
 \|f\|_{L^2}\lesssim \varepsilon_12^{-k}2^{p_0m},\qquad \|g\|_{L^\infty}\lesssim\varepsilon_12^{-m/2},\qquad \|h\|_{L^2}\lesssim \varepsilon_12^{l_2/2}2^{-10k_+}.
\end{equation*}
It follows from \eqref{bn32} that
\begin{equation*}
 \|\mathcal{F}^{-1}(m_1)\|_{L^1}\lesssim 2^{-m}2^{-l_2}2^{3k/2}.
\end{equation*}
Therefore, using Lemma \ref{touse} and recalling that $2^{-l_2/2}\lesssim 2^{m/4}2^{-3k/8}$ and that $k\leq m/10$,
\begin{equation*}
 |F_{l_1,l_2,1}(\xi,s)|\lesssim \varepsilon_1^32^{-k}2^{p_0m}\cdot 2^{-m/2}\cdot 2^{l_2/2}2^{-10k_+}\cdot 2^{-m}2^{-l_2}2^{3k/2}
\lesssim \varepsilon_1^32^{-10k_+}2^{-m}\cdot 2^{-m/8}.
\end{equation*}

A similar argument shows that $|G_{l_1,l_2,1}(\xi,s)|\lesssim \varepsilon_1^32^{-10k_+}2^{-m}\cdot 2^{-m/8}$. Therefore, for \eqref{bn30} it suffices to prove that
\begin{equation}\label{bn35}
|J_{l_1,l_2,1}(\xi,s)|\lesssim \varepsilon_1^32^{-m}2^{-3p_1m}2^{-10k_+}.
\end{equation}

For this we integrate by parts again in $\eta$ and estimate
\begin{equation*}
|J_{l_1,l_2,1}(\xi,s)|\leq |J_{l_1,l_2,2}(\xi,s)|+|F_{l_1,l_2,2}(\xi,s)|+|G_{l_1,l_2,2}(\xi,s)|,
\end{equation*}
where
\begin{equation*}
\begin{split}
&J_{l_1,l_2,2}(\xi,s):=\int_{\mathbb{R}^2}e^{is\Phi(\xi,\eta,\sigma)}\widehat{f_{k_1}^+}(\xi+\eta,s)\widehat{f_{k_2}^+}(\xi+\sigma,s)
\widehat{f_{k_3}^-}(-\xi-\eta-\sigma,s)(\partial_\eta m_2)(\eta,\sigma)\,d\eta d\sigma,\\
&F_{l_1,l_2,2}(\xi,s):=\int_{\mathbb{R}^2}e^{is\Phi(\xi,\eta,\sigma)}(\partial\widehat{f_{k_1}^+})(\xi+\eta,s)\widehat{f_{k_2}^+}(\xi+\sigma,s)
\widehat{f_{k_3}^-}(-\xi-\eta-\sigma,s)m_2(\eta,\sigma)\,d\eta d\sigma,\\
&G_{l_1,l_2,2}(\xi,s):=\int_{\mathbb{R}^2}e^{is\Phi(\xi,\eta,\sigma)}\widehat{f_{k_1}^+}(\xi+\eta,s)\widehat{f_{k_2}^+}(\xi+\sigma,s)
(\partial\widehat{f_{k_3}^-})(-\xi-\eta-\sigma,s)m_2(\eta,\sigma)\,d\eta d\sigma.
\end{split}
\end{equation*}
and
\begin{equation*}
 m_2(\eta,\sigma):=\frac{(\partial_\eta m_1)(\eta,\sigma)}{s(\partial_\eta\Phi)(\xi,\eta,\sigma)}.
\end{equation*}
It follows from \eqref{bn32} that $m_2$ satisfies the stronger symbol-type bounds
\begin{equation}\label{bn44}
|(\partial_\eta^a\partial_\sigma^bm_2)(\eta,\sigma)|\lesssim (2^{-m}2^{-l_1-l_2}2^{3k/2})(2^{-m}2^{-l_2}2^{3k/2})(2^{-al_1}2^{-bl_2})\cdot 
\mathbf{1}_{[0,2^{l_1+4}]}(|\eta|)\mathbf{1}_{[2^{l_2-4},2^{l_2+4}]}(|\sigma|),
\end{equation}
for $a,b\in[0,19]\cap\mathbb{Z}$. Therefore, using Lemma \ref{touse} as before,
\begin{equation*}
 |F_{l_1,l_2,2}(\xi,s)|+|G_{l_1,l_2,2}(\xi,s)|\lesssim \varepsilon_1^32^{-10k_+}2^{-m}\cdot 2^{-m/8}.
\end{equation*}
Moreover, we can now estimate $|J_{l_1,l_2,2}(\xi,s)|$ using only \eqref{bn44} and the $L^\infty$ bounds in the last line of \eqref{bn13},
\begin{equation*}
|J_{l_1,l_2,2}(\xi,s)|\lesssim 2^{l_1+l_2}\cdot\varepsilon_1^32^{-30k_+}\cdot (2^{-m}2^{-l_1-l_2}2^{3k/2})^2\lesssim \varepsilon_1^32^{-10k_+}2^{-m}\cdot 2^{-m/50}.
\end{equation*}
This completes the proof of \eqref{bn35} and \eqref{bn30}.

A similar argument shows that
\begin{equation*}
|J_{l_1,l_2}(\xi,s)|\lesssim 2^{-m}\varepsilon_1^22^{-3p_1m}2^{-10k_+},\qquad\text{ if }l_1\geq\max(l_2,\overline{l}+1).
\end{equation*}

{\bf{Step 2.}} Using the decomposition \eqref{bn27}, for \eqref{bn28} it suffices to prove that
\begin{equation}\label{bn50}
 \Big|J_{\overline{l},\overline{l}}(\xi,s)-\widetilde{c}\frac{|\xi|^{3/2}\widehat{f_{k_1}^+}(\xi,s)\widehat{f_{k_2}^+}(\xi,s)
\widehat{f_{k_3}^-}(-\xi,s)}{s+1}\Big|\lesssim 2^{-m}\varepsilon_1^32^{-2p_1m}2^{-10k_+}.
\end{equation}

To prove \eqref{bn50} we notice that
\begin{equation*}
\Big|\Phi(\xi,\eta,\sigma)+\frac{\eta\sigma}{4|\xi|^{3/2}}\Big|\lesssim 2^{-5k/2}(|\eta|+|\sigma|)^3,
\end{equation*}
as long as $|\eta|+|\sigma|\leq 2^{k-5}$. Therefore, using the $L^\infty$ bounds in the last line of \eqref{bn13}
\begin{equation}\label{bn51}
 \Big|J_{\overline{l},\overline{l}}(\xi,s)-J'_{\overline{l},\overline{l}}(\xi,s)\Big|\lesssim \varepsilon_1^32^m2^{-5k/2}2^{5\overline{l}}2^{-30k_+}
\lesssim \varepsilon_1^32^{-5m/4}2^{-10k_+}.
\end{equation}
where
\begin{equation}\label{bn52}
J'_{\overline{l},\overline{l}}(\xi,s):=\int_{\mathbb{R}^2}e^{-is\eta\sigma/(4|\xi|^{3/2})}\widehat{f_{k_1}^+}(\xi+\eta,s)\widehat{f_{k_2}^+}(\xi+\sigma,s)
\widehat{f_{k_3}^-}(-\xi-\eta-\sigma,s)\varphi(2^{-\overline{l}}\eta)\varphi(2^{-\overline{l}}\sigma)\,d\eta d\sigma.
\end{equation}
Moreover, using the bounds in the last two lines of \eqref{bn13}
\begin{equation*}
|\widehat{f_{k_1}^+}(\xi+\eta,s)\widehat{f_{k_2}^+}(\xi+\sigma,s)\widehat{f_{k_3}^-}(-\xi-\eta-\sigma,s)-
\widehat{f_{k_1}^+}(\xi,s)\widehat{f_{k_2}^+}(\xi,s)\widehat{f_{k_3}^-}(-\xi,s)|\lesssim \varepsilon_1^32^{\overline{l}/2}\cdot 2^{-20k_+}2^{p_0m}2^{-k},
\end{equation*}
whenever $|\eta|+|\sigma|\leq 2^{\overline{l}+4}$. Therefore
\begin{equation}\label{bn53}
\begin{split}
\Big|J'_{\overline{l},\overline{l}}(\xi,s)-&\int_{\mathbb{R}^2}e^{-is\eta\sigma/(4|\xi|^{3/2})}\widehat{f_{k_1}^+}(\xi,s)\widehat{f_{k_2}^+}(\xi,s)
\widehat{f_{k_3}^-}(-\xi,s)\varphi(2^{-\overline{l}}\eta)\varphi(2^{-\overline{l}}\sigma)\,d\eta d\sigma\Big|\\
&\lesssim 2^{2\overline{l}}\varepsilon_1^32^{\overline{l}/2}\cdot 2^{-20k_+}2^{p_0m}2^{-k}\\
&\lesssim \varepsilon_1^32^{-9m/8}2^{-10k_+}.
\end{split}
\end{equation}

Starting from the general formula
\begin{equation*}
\int_{\mathbb{R}}e^{-ax^2-bx}\,dx=e^{b^2/(4a)}\sqrt{\pi}/\sqrt{a},\qquad a,b\in\mathbb{C},\,\,\Re\,a>0,
\end{equation*}
we calculate, for any $N\geq 1$,
\begin{equation*}
\int_{\mathbb{R}\times\mathbb{R}}e^{-ixy}e^{-x^2/N^2}e^{-y^2/N^2}\,dxdy=\sqrt{\pi}N\int_{\mathbb{R}}e^{-y^2/N^2}e^{-N^2y^2/4}\,dy=2\pi+O(N^{-1}).
\end{equation*}
Therefore, for $N\geq 1$,
\begin{equation*}
\int_{\mathbb{R}\times\mathbb{R}}e^{-ixy}\varphi(x/N)\varphi(y/N)\,dxdy=2\pi+O(N^{-1/2}).
\end{equation*}
Recalling also that $2^{\overline{l}}\approx |\xi|^{3/4}2^{-49m/100}$, it follows that
\begin{equation*}
\Big|\int_{\mathbb{R}^2}e^{-is\eta\sigma/(4|\xi|^{3/2})}\varphi(2^{-\overline{l}}\eta)\varphi(2^{-\overline{l}}\sigma)\,d\eta d\sigma-\frac{4|\xi|^{3/2}}{s}(2\pi)\Big|\lesssim 2^{3k/2}2^{-(1+2p_1)m}.
\end{equation*}
Therefore, using also \eqref{bn13},
\begin{equation}\label{bn54}
\begin{split}
\Big|\int_{\mathbb{R}^2}e^{-is\eta\sigma/(4|\xi|^{3/2})}&\widehat{f_{k_1}^+}(\xi,s)\widehat{f_{k_2}^+}(\xi,s)
\widehat{f_{k_3}^-}(-\xi,s)\varphi(2^{-\overline{l}}\eta)\varphi(2^{-\overline{l}}\sigma)\,d\eta d\sigma\\
&-\frac{8\pi|\xi|^{3/2}\widehat{f_{k_1}^+}(\xi,s)\widehat{f_{k_2}^+}(\xi,s)
\widehat{f_{k_3}^-}(-\xi,s)}{s}\Big|\lesssim \varepsilon_1^32^{-(1+2p_1)m}2^{-10k_+},
\end{split}
\end{equation}
and the bound \eqref{bn50} follows from \eqref{bn51}, \eqref{bn53}, and \eqref{bn54}.
\end{proof}

\begin{lem}\label{bb10}
The bounds \eqref{bn20} hold provided that
\begin{equation}\label{bn60}
\min(k_1,k_2,k_3)+\mathrm{med}(k_1,k_2,k_3)\leq -51m/50\quad\text{ and }\quad\max(|k_1-k|,|k_2-k|,|k_3-k|)\geq 21. 
\end{equation}
The bounds \eqref{bn21} hold provided that
\begin{equation}\label{bn61}
\min(k_1,k_2,k_3)+\mathrm{med}(k_1,k_2,k_3)\leq -51m/50.
\end{equation}
\end{lem}

\begin{proof}[Proof of Lemma \ref{bb10}]
 Using only the $L^\infty$ bounds in the last line of \eqref{bn13} we estimate, for any $(\iota_1,\iota_2,\iota_3)\in\{(+,+,-),(+,+,+),(+,-,-),(-,-,-)\}$,
\begin{equation*}
\begin{split}
\Big|\int_{t_1}^{t_2} e^{iH(\xi,s)}I_{k_1,k_2,k_3}^{\iota_1,\iota_2,\iota_3}(\xi,s)\,ds\Big|
\lesssim 2^m\sup_{s\in[t_1,t_2]}\int_{\mathbb{R}\times\mathbb{R}}
|\widehat{f_{k_1}^{\iota_1}}(\xi-\eta,s)|\,|\widehat{f_{k_2}^{\iota_2}}(\eta-\sigma,s)|\,|\widehat{f_{k_3}^{\iota_3}}(\sigma,s)|\,d\eta d\sigma\\
\lesssim 2^m 2^{\min(k_1,k_2,k_3)}2^{\mathrm{med}(k_1,k_2,k_3)}\cdot\varepsilon_1^32^{-10k_+}, 
\end{split}
\end{equation*}
which clearly suffices in view of the assumptions \eqref{bn60} and \eqref{bn61}.
\end{proof}

\begin{lem}\label{bb11}
The bounds \eqref{bn20} hold provided that
\begin{equation}\label{bn70}
\begin{split}
&\max(|k_1-k|,|k_2-k|,|k_3-k|)\geq 21,\\
&\min(k_1,k_2,k_3)\geq -19m/20,\qquad\max(|k_1-k_3|,|k_2-k_3|)\geq 5.
\end{split}
\end{equation}
\end{lem}

\begin{proof}[Proof of Lemma \ref{bb11}] Recall the definition
\begin{equation}\label{bn71}
I_{k_1,k_2,k_3}^{+,+,-}(\xi,s)=\int_{\mathbb{R}^2}e^{is\Phi(\xi,\eta,\sigma)}\widehat{f_{k_1}^+}(\xi+\eta,s)\widehat{f_{k_2}^+}(\xi+\sigma,s)
\widehat{f_{k_3}^-}(-\xi-\eta-\sigma,s)\,d\eta d\sigma,
\end{equation}
where
\begin{equation*}
\Phi(\xi,\eta,\sigma)=\Lambda(\xi)-\Lambda(\xi+\eta)-\Lambda(\xi+\sigma)+\Lambda(\xi+\eta+\sigma).
\end{equation*}
It suffices to prove that, for any $s\in[t_1,t_2]$,
\begin{equation}\label{bn72}
\big|I_{k_1,k_2,k_3}^{+,+,-}(\xi,s)\big|\lesssim 2^{-m}\varepsilon_1^32^{-2p_1m}2^{-10k_+}.
\end{equation}

By symmetry, we may assume that $|k_1-k_3|\geq 5$ and notice that
\begin{equation}\label{bn73}
|(\partial_\eta\Phi)(\xi,\eta,\sigma)|=|-\Lambda'(\xi+\eta)+\Lambda'(\xi+\eta+\sigma)|\gtrsim 2^{-\min(k_1,k_3)/2},
\end{equation}
provided that $|\xi+\eta|\in[2^{k_1-2},2^{k_1+2}]$, $|\xi+\eta+\sigma|\in[2^{k_3-2},2^{k_3+2}]$. As in the proof of Lemma \ref{bb2}, we integrate by parts in $\eta$ to
estimate
\begin{equation*}
|I_{k_1,k_2,k_3}^{+,+,-}(\xi,s)|\leq |J_{1}(\xi,s)|+|F_{1}(\xi,s)|+|G_{1}(\xi,s)|,
\end{equation*}
where
\begin{equation*}
\begin{split}
&J_{1}(\xi,s):=\int_{\mathbb{R}^2}e^{is\Phi(\xi,\eta,\sigma)}\widehat{f_{k_1}^+}(\xi+\eta,s)\widehat{f_{k_2}^+}(\xi+\sigma,s)
\widehat{f_{k_3}^-}(-\xi-\eta-\sigma,s)(\partial_\eta m_3)(\eta,\sigma)\,d\eta d\sigma,\\
&F_{1}(\xi,s):=\int_{\mathbb{R}^2}e^{is\Phi(\xi,\eta,\sigma)}(\partial\widehat{f_{k_1}^+})(\xi+\eta,s)\widehat{f_{k_2}^+}(\xi+\sigma,s)
\widehat{f_{k_3}^-}(-\xi-\eta-\sigma,s)m_3(\eta,\sigma)\,d\eta d\sigma,\\
&G_{1}(\xi,s):=\int_{\mathbb{R}^2}e^{is\Phi(\xi,\eta,\sigma)}\widehat{f_{k_1}^+}(\xi+\eta,s)\widehat{f_{k_2}^+}(\xi+\sigma,s)
(\partial\widehat{f_{k_3}^-})(-\xi-\eta-\sigma,s)m_3(\eta,\sigma)\,d\eta d\sigma,
\end{split}
\end{equation*}
and
\begin{equation*}
 m_3(\eta,\sigma):=\frac{1}{s(\partial_\eta\Phi)(\xi,\eta,\sigma)}\cdot 
\varphi_{[k_1-1,k_1+1]}(\xi+\eta)\varphi_{[k_3-1,k_3+1]}(\xi+\eta+\sigma).
\end{equation*}
Using also \eqref{bn73}, it follows easily that
\begin{equation*}
 \|\mathcal{F}^{-1}(m_3)\|_{L^1}\lesssim 2^{-m}2^{\min(k_1,k_3)/2}, \qquad \|\mathcal{F}^{-1}(\partial_\eta m_3)\|_{L^1}\lesssim 2^{-m}2^{-\min(k_1,k_3)/2}.
\end{equation*}
We apply first Lemma \ref{touse} with
\begin{equation*}
\widehat{f}(\eta):=e^{-is\Lambda(\xi+\eta)}\widehat{f_{k_1}^+}(\xi+\eta,s),\,
\widehat{g}(\sigma):=e^{-is\Lambda(\xi+\sigma)}\widehat{f_{k_2}^+}(\xi+\sigma,s),\,
\widehat{h}(\theta):=e^{is\Lambda(\xi-\theta)}\widehat{f_{k_3}^-}(-\xi+\theta,s).
\end{equation*}
Using also \eqref{bn13} and \eqref{bn13.5} we conclude that
\begin{equation*}
 |J_{1}(\xi,s)|\lesssim \varepsilon_1^32^{-m/2}2^{2p_0m}2^{-N_0\max(k_1,k_2,k_3)_+}\cdot 2^{-m}2^{-\min(k_1,k_3)/2}.
\end{equation*}
Similarly, we apply Lemma \ref{touse} with 
\begin{equation*}
\widehat{f}(\eta):=e^{-is\Lambda(\xi+\eta)}(\partial\widehat{f_{k_1}^+})(\xi+\eta,s),\,
\widehat{g}(\sigma):=e^{-is\Lambda(\xi+\sigma)}\widehat{f_{k_2}^+}(\xi+\sigma,s),\,
\widehat{h}(\theta):=e^{is\Lambda(\xi-\theta)}\widehat{f_{k_3}^-}(-\xi+\theta,s),
\end{equation*}
and use \eqref{bn13} and \eqref{bn13.5} we conclude that
\begin{equation*}
 |F_{1}(\xi,s)|\lesssim \varepsilon_1^32^{-m/2}2^{2p_0m}2^{-k_1}2^{-N_0\max(k_2,k_3)_+}\cdot 2^{-m}2^{\min(k_1,k_3)/2}.
\end{equation*}
Finally, we apply Lemma \ref{touse} with 
\begin{equation*}
\widehat{f}(\eta):=e^{-is\Lambda(\xi+\eta)}\widehat{f_{k_1}^+}(\xi+\eta,s),\,
\widehat{g}(\sigma):=e^{-is\Lambda(\xi+\sigma)}\widehat{f_{k_2}^+}(\xi+\sigma,s),\,
\widehat{h}(\theta):=e^{is\Lambda(\xi-\theta)}(\partial\widehat{f_{k_3}^-})(-\xi+\theta,s),
\end{equation*}
and use \eqref{bn13} and \eqref{bn13.5} we conclude that
\begin{equation*}
 |G_{1}(\xi,s)|\lesssim \varepsilon_1^32^{-m/2}2^{2p_0m}2^{-k_3}2^{-N_0\max(k_2,k_1)_+}\cdot 2^{-m}2^{\min(k_1,k_3)/2}.
\end{equation*}
Therefore
\begin{equation*}
 |J_{1}(\xi,s)|+|F_{1}(\xi,s)|+|G_{1}(\xi,s)|\lesssim\varepsilon_1^32^{-m}2^{-10k_+}2^{-m(1/2-2p_0)}(2^{-\min(k_1,k_3)/2}+2^{10\max(k_1,k_2,k_3)_+}),
\end{equation*}
and the desired bound \eqref{bn72} follows from the assumptions $-\min(k_1,k_3)/2\leq 19m/40$, see \eqref{bn70}, and $10\max(k_1,k_2,k_3)_+\leq m/5$ 
(see the hypothesis of Lemma \ref{bb1}).
\end{proof}

\begin{lem}\label{bb12}
The bounds \eqref{bn20} hold provided that
\begin{equation}\label{bn80}
\begin{split}
&\max(|k_1-k|,|k_2-k|,|k_3-k|)\geq 21,\\
&\min(k_1,k_2,k_3)\geq -19m/20,\qquad\max(|k_1-k_3|,|k_2-k_3|)\leq 4.
\end{split}
\end{equation}
\end{lem}

\begin{proof}[Proof of Lemma \ref{bb12}] We may assume that
\begin{equation}\label{bn81}
 \min(k_1,k_2,k_3)\geq k+10,
\end{equation}
and rewrite
\begin{equation}\label{bn82}
I_{k_1,k_2,k_3}^{+,+,-}(\xi,s)=\int_{\mathbb{R}^2}e^{is\Phi(\xi,\eta,\sigma)}\widehat{f_{k_1}^+}(\xi+\eta,s)\widehat{f_{k_2}^+}(\xi+\sigma,s)
\widehat{f_{k_3}^-}(-\xi-\eta-\sigma,s)\varphi_{[k_2-4,k_2+4]}(\sigma)\,d\eta d\sigma,
\end{equation}
where, as before,
\begin{equation*}
\Phi(\xi,\eta,\sigma)=\Lambda(\xi)-\Lambda(\xi+\eta)-\Lambda(\xi+\sigma)+\Lambda(\xi+\eta+\sigma).
\end{equation*}
It suffices to prove that, for any $s\in[t_1,t_2]$,
\begin{equation}\label{bn83}
\big|I_{k_1,k_2,k_3}^{+,+,-}(\xi,s)\big|\lesssim 2^{-m}\varepsilon_1^32^{-2p_1m}2^{-10k_+}.
\end{equation}

We argue as in the proof of Lemma \ref{bb11}. Notice that
\begin{equation}\label{bn84}
|(\partial_\eta\Phi)(\xi,\eta,\sigma)|=|-\Lambda'(\xi+\eta)+\Lambda'(\xi+\eta+\sigma)|\gtrsim 2^{-k_2/2},
\end{equation}
provided that $|\xi+\eta|\in[2^{k_1-2},2^{k_1+2}]$, $|\xi+\eta+\sigma|\in[2^{k_3-2},2^{k_3+2}]$, and $|\sigma|\approx 2^{k_2}$ (recall also that $2^{k_1}\approx 2^{k_2}\approx 2^{k_3}$). As before, we integrate by parts in $\eta$ to
estimate
\begin{equation*}
|I_{k_1,k_2,k_3}^{+,+,-}(\xi,s)|\leq |J_{2}(\xi,s)|+|F_{2}(\xi,s)|+|G_{2}(\xi,s)|,
\end{equation*}
where
\begin{equation*}
\begin{split}
&J_{2}(\xi,s):=\int_{\mathbb{R}^2}e^{is\Phi(\xi,\eta,\sigma)}\widehat{f_{k_1}^+}(\xi+\eta,s)\widehat{f_{k_2}^+}(\xi+\sigma,s)
\widehat{f_{k_3}^-}(-\xi-\eta-\sigma,s)(\partial_\eta m_4)(\eta,\sigma)\,d\eta d\sigma,\\
&F_{2}(\xi,s):=\int_{\mathbb{R}^2}e^{is\Phi(\xi,\eta,\sigma)}(\partial\widehat{f_{k_1}^+})(\xi+\eta,s)\widehat{f_{k_2}^+}(\xi+\sigma,s)
\widehat{f_{k_3}^-}(-\xi-\eta-\sigma,s)m_4(\eta,\sigma)\,d\eta d\sigma,\\
&G_{2}(\xi,s):=\int_{\mathbb{R}^2}e^{is\Phi(\xi,\eta,\sigma)}\widehat{f_{k_1}^+}(\xi+\eta,s)\widehat{f_{k_2}^+}(\xi+\sigma,s)
(\partial\widehat{f_{k_3}^-})(-\xi-\eta-\sigma,s)m_4(\eta,\sigma)\,d\eta d\sigma,
\end{split}
\end{equation*}
and
\begin{equation*}
 m_4(\eta,\sigma):=\frac{1}{s(\partial_\eta\Phi)(\xi,\eta,\sigma)}\cdot 
\varphi_{[k_1-1,k_1+1]}(\xi+\eta)\varphi_{[k_3-1,k_3+1]}(\xi+\eta+\sigma)\varphi_{[k_2-4,k_2+4]}(\sigma).
\end{equation*}
Using also \eqref{bn84}, it follows easily that
\begin{equation*}
 \|\mathcal{F}^{-1}(m_4)\|_{L^1}\lesssim 2^{-m}2^{k_2/2}, \qquad \|\mathcal{F}^{-1}(\partial_\eta m_4)\|_{L^1}\lesssim 2^{-m}2^{-k_2/2}.
\end{equation*}
We apply first Lemma \ref{touse} with
\begin{equation*}
\widehat{f}(\eta):=e^{-is\Lambda(\xi+\eta)}\widehat{f_{k_1}^+}(\xi+\eta,s),\,
\widehat{g}(\sigma):=e^{-is\Lambda(\xi+\sigma)}\widehat{f_{k_2}^+}(\xi+\sigma,s),\,
\widehat{h}(\theta):=e^{is\Lambda(\xi-\theta)}\widehat{f_{k_3}^-}(-\xi+\theta,s).
\end{equation*}
Using also \eqref{bn13} and \eqref{bn13.5} we conclude that
\begin{equation*}
 |J_{2}(\xi,s)|\lesssim \varepsilon_1^32^{-m/2}2^{2p_0m}2^{-N_0\max(k_1,k_2,k_3)_+}\cdot 2^{-m}2^{-k_2/2}.
\end{equation*}
Similarly, we apply Lemma \ref{touse} with 
\begin{equation*}
\widehat{f}(\eta):=e^{-is\Lambda(\xi+\eta)}(\partial\widehat{f_{k_1}^+})(\xi+\eta,s),\,
\widehat{g}(\sigma):=e^{-is\Lambda(\xi+\sigma)}\widehat{f_{k_2}^+}(\xi+\sigma,s),\,
\widehat{h}(\theta):=e^{is\Lambda(\xi-\theta)}\widehat{f_{k_3}^-}(-\xi+\theta,s),
\end{equation*}
and use \eqref{bn13} and \eqref{bn13.5} we conclude that
\begin{equation*}
 |F_{2}(\xi,s)|\lesssim \varepsilon_1^32^{-m/2}2^{2p_0m}2^{-k_1}2^{-N_0\max(k_2,k_3)_+}\cdot 2^{-m}2^{k_2/2}.
\end{equation*}
Finally, we apply Lemma \ref{touse} with 
\begin{equation*}
\widehat{f}(\eta):=e^{-is\Lambda(\xi+\eta)}\widehat{f_{k_1}^+}(\xi+\eta,s),\,
\widehat{g}(\sigma):=e^{-is\Lambda(\xi+\sigma)}\widehat{f_{k_2}^+}(\xi+\sigma,s),\,
\widehat{h}(\theta):=e^{is\Lambda(\xi-\theta)}(\partial\widehat{f_{k_3}^-})(-\xi+\theta,s),
\end{equation*}
and use \eqref{bn13} and \eqref{bn13.5} we conclude that
\begin{equation*}
 |G_{2}(\xi,s)|\lesssim \varepsilon_1^32^{-m/2}2^{2p_0m}2^{-k_3}2^{-N_0\max(k_2,k_1)_+}\cdot 2^{-m}2^{k_2/2}.
\end{equation*}
Recalling the assumption $2^{k_1}\approx 2^{k_2}\approx 2^{k_3}$, it follows that
\begin{equation*}
 |J_{1}(\xi,s)|+|F_{1}(\xi,s)|+|G_{1}(\xi,s)|\lesssim\varepsilon_1^32^{-m}2^{-N_0\max(k_2,0)}2^{-m(1/2-2p_0)}2^{-k_2/2},
\end{equation*}
and the desired bound \eqref{bn83} follows from the assumptions $-k_2/2\leq 19m/40$.
 \end{proof}
 
\begin{lem}\label{bb13}
The bounds \eqref{bn20} hold provided that
\begin{equation}\label{bn90}
\begin{split}
&\max(|k_1-k|,|k_2-k|,|k_3-k|)\geq 21,\\
&\min(k_1,k_2,k_3)\leq -19m/20,\qquad\min(k_1,k_2,k_3)+\mathrm{med}(k_1,k_2,k_3)\geq -51m/50.
\end{split}
\end{equation}
\end{lem}

\begin{proof}[Proof of Lemma \ref{bb13}] In this case we cannot prove pointwise bounds on $\big|I_{k_1,k_2,k_3}^{+,+,-}(\xi,s)|$ and we need to integrate by parts in $s$. The desired bound \eqref{bn20} is equivalent to
\begin{equation}\label{bn91}
\Big|\int_{\mathbb{R}^2\times[t_1,t_2]}e^{iH(\xi,s)}e^{is\Phi(\xi,\eta,\sigma)}\widehat{f_{k_1}^+}(\xi+\eta,s)\widehat{f_{k_2}^+}(\xi+\sigma,s)
\widehat{f_{k_3}^-}(-\xi-\eta-\sigma,s)\,d\eta d\sigma ds\Big|\lesssim \varepsilon_1^32^{-2p_1m}2^{-10k_+},
\end{equation}
where
\begin{equation*}
\begin{split}
&\Phi(\xi,\eta,\sigma)=\Lambda(\xi)-\Lambda(\xi+\eta)-\Lambda(\xi+\sigma)+\Lambda(\xi+\eta+\sigma),\\
&H(\xi,s)=\frac{2c_0}{\pi}|\xi|^{3/2}\int_0^s|\widehat{f}(\xi,r)|^2 \frac{dr}{r+1}.
\end{split}
\end{equation*}
We consider two cases.

{\bf {Case 1:}} $k_3=\min(k_1,k_2,k_3)$. In this case, recalling also the assumption $k_1,k_2,k_3\in[-4m,m/50-1000]$, we have
\begin{equation}\label{bn92}
k_3\in[-21m/20,-19m/20],\qquad k_1,k_2\in[-m/10,m/50-1000].
\end{equation}
We see easily that
\begin{equation}\label{bn93}
-\Phi(\xi,\eta,\sigma)\geq 2^{\min(k_1,k_2)/2-100},
\end{equation}
provided that $|\xi+\eta|\in[2^{k_1-2},2^{k_1+2}]$, $|\xi+\sigma|\in[2^{k_2-2},2^{k_2+2}]$, and $|\xi+\eta+\sigma|\in[2^{k_3-2},2^{k_3+2}]$. Letting $\dot{H}(\xi,s):=(\partial_sH)(\xi,s)$, we notice that
\begin{equation}\label{bn93.5}
|\dot{H}(\xi,s)|\lesssim \varepsilon_1^22^{3k/2}2^{-20k_+}2^{-m}.
\end{equation}

We integrate by parts in $s$ to conclude that the integral in the left-hand side of \eqref{bn91} is dominated by
\begin{equation}\label{bn94}
\begin{split}
&\int_{t_1}^{t_2}\Big|\int_{\mathbb{R}^2}e^{is\Phi(\xi,\eta,\sigma)}\frac{d}{ds}\Big[\frac{1}{\Phi(\xi,\eta,\sigma)+\dot{H}(\xi,s)}\widehat{f_{k_1}^+}(\xi+\eta,s)\widehat{f_{k_2}^+}(\xi+\sigma,s)
\widehat{f_{k_3}^-}(-\xi-\eta-\sigma,s)\Big]\,d\eta d\sigma\Big| ds\\
&+\sum_{j=1}^2\Big|\int_{\mathbb{R}^2}e^{it_j\Phi(\xi,\eta,\sigma)}\cdot\frac{1}{\Phi(\xi,\eta,\sigma)+\dot{H}(\xi,t_j)}\widehat{f_{k_1}^+}(\xi+\eta,t_j)\widehat{f_{k_2}^+}(\xi+\sigma,t_j)
\widehat{f_{k_3}^-}(-\xi-\eta-\sigma,t_j)\,d\eta d\sigma\Big|\\
&:=B^0(\xi)+\sum_{j=1}^2B_j(\xi).
\end{split}
\end{equation}

Let
\begin{equation*}
 m_5(\eta,\sigma):=\frac{1}{\Phi(\xi,\eta,\sigma)+\dot{H}(\xi,s)}\cdot 
\varphi_{[k_1-1,k_1+1]}(\xi+\eta)\varphi_{[k_2-1,k_2+1]}(\xi+\sigma)\varphi_{[k_3-1,k_3+1]}(\xi+\eta+\sigma).
\end{equation*}
Using \eqref{bn92}--\eqref{bn93.5} and integration by parts it is easy to see that, for any $s\in[t_1,t_2]$,
\begin{equation}\label{bn95}
\|\mathcal{F}^{-1}(m_5)\|_{L^1}\lesssim 2^{-\min(k_1,k_2)/2}.
\end{equation}
Using the $L^\infty$ bound in \eqref{bn13} we estimate, for $j\in\{1,2\}$,
\begin{equation}\label{bn96}
B_j(\xi)\lesssim \varepsilon_1^3\|m_5\|_{L^1}\lesssim \varepsilon_1^32^{k_3}2^{\min(k_1,k_2)/2}.
\end{equation}

Now we estimate
\begin{equation}\label{bn97}
\begin{split}
&B^0(\xi)\lesssim 2^m\sup_{s\in[t_1,t_2]}[B^0_0(\xi,s)+B^0_1(\xi,s)+B^0_2(\xi,s)+B^0_3(\xi,s)],\\  &B^0_0(\xi,s):=\int_{\mathbb{R}^2}\Big|\frac{(\partial_s\dot{H})(\xi,s)}{(\Phi(\xi,\eta,\sigma)+\dot{H}(\xi,s))^2}\widehat{f_{k_1}^+}(\xi+\eta,s)\widehat{f_{k_2}^+}(\xi+\sigma,s)
\widehat{f_{k_3}^-}(-\xi-\eta-\sigma,s)\Big|\,d\eta d\sigma,\\
&B^0_1(\xi,s):=\Big|\int_{\mathbb{R}^2}e^{is\Phi(\xi,\eta,\sigma)}m_5(\eta,\sigma)(\partial_s\widehat{f_{k_1}^+})(\xi+\eta,s)\widehat{f_{k_2}^+}(\xi+\sigma,s)
\widehat{f_{k_3}^-}(-\xi-\eta-\sigma,s)\,d\eta d\sigma\Big|,\\
&B^0_2(\xi,s):=\Big|\int_{\mathbb{R}^2}e^{is\Phi(\xi,\eta,\sigma)}m_5(\eta,\sigma)\widehat{f_{k_1}^+}(\xi+\eta,s)(\partial_s\widehat{f_{k_2}^+})(\xi+\sigma,s)
\widehat{f_{k_3}^-}(-\xi-\eta-\sigma,s)\,d\eta d\sigma\Big|,\\
&B^0_3(\xi,s):=\Big|\int_{\mathbb{R}^2}e^{is\Phi(\xi,\eta,\sigma)}m_5(\eta,\sigma)\widehat{f_{k_1}^+}(\xi+\eta,s)\widehat{f_{k_2}^+}(\xi+\sigma,s)
(\partial_s\widehat{f_{k_3}^-})(-\xi-\eta-\sigma,s)\,d\eta d\sigma\Big|.
\end{split}
\end{equation}

As before, we combine Lemma \ref{touse}, \eqref{bn95}, and the bounds \eqref{bn13}, \eqref{bn13.5}, and \eqref{touse4} to conclude that
\begin{equation}\label{bn97.1}
\begin{split}
\sup_{s\in[t_1,t_2]}[B^0_1(\xi,s)+B^0_2(\xi,s)+B^0_3(\xi,s)]&\lesssim 2^{-\min(k_1,k_2)/2}\cdot\varepsilon_1 2^{-m/2}\cdot \varepsilon_12^{\min(k_1,k_2)/2}\cdot \varepsilon_1 2^{(3p_0-1)m}\\
&\lesssim \varepsilon_1^32^{(3p_0-3/2)m}.
\end{split}
\end{equation}
In addition, using \eqref{touse4.1} and the definition of the function $H$, we have
\begin{equation}\label{bn97.15}
\sup_{s\in[t_1,t_2]}|(\partial_s\dot{H})(\xi,s)|\lesssim \varepsilon_1^22^{3k/2}2^{-20k_+}2^{(3p_0-3/2)m}.
\end{equation}
Therefore
\begin{equation}\label{bn97.2}
\sup_{s\in[t_1,t_2]}B^0_0(\xi,s)\lesssim \varepsilon_1^32^{3k/2}2^{-20k_+}2^{(3p_0-3/2)m}\cdot 2^{k_3}.
\end{equation}

We combine now \eqref{bn94}, \eqref{bn96}, \eqref{bn97}, \eqref{bn97.1}, and \eqref{bn97.2} to conclude that
\begin{equation*}
\Big|\int_{\mathbb{R}^2\times[t_1,t_2]}e^{iH(\xi,s)}e^{is\Phi(\xi,\eta,\sigma)}\widehat{f_{k_1}^+}(\xi+\eta,s)\widehat{f_{k_2}^+}(\xi+\sigma,s)
\widehat{f_{k_3}^-}(-\xi-\eta-\sigma,s)\,d\eta d\sigma ds\Big|\lesssim \varepsilon_1^32^{(3p_0-1/2)m},
\end{equation*}
which is clearly stronger than the desired bound \eqref{bn91}.

{\bf {Case 2:}} $k_3\neq\min(k_1,k_2,k_3)$. By symmetry we may assume $k_1=\min(k_1,k_2,k_3)$. Recalling also the assumption $k_1,k_2,k_3\in[-4m,m/50-1000]$, we have
\begin{equation}\label{alo1}
k_1\in[-21m/20,-19m/20],\qquad k_2,k_3\in[-m/10,m/50-1000].
\end{equation}
Recalling the restriction $|\xi|\in[2^{k-1},2^{k+1}]$, we define
\begin{equation*}
\chi_{k,m}(\eta):=
\begin{cases}
1&\text{ if }k\geq k_1+11,\\
1-\varphi(2^{11m/10}\eta)&\text{ if }k\leq k_1+10,
\end{cases}
\end{equation*}
and notice that, as a consequence of the $L^\infty$ bound in \eqref{bn13},
\begin{equation*}
\begin{split}
\Big|\int_{\mathbb{R}^2\times[t_1,t_2]}(1-\chi_{k,m}(\eta))e^{iH(\xi,s)}e^{is\Phi(\xi,\eta,\sigma)}\widehat{f_{k_1}^+}(\xi+\eta,s)\widehat{f_{k_2}^+}(\xi+\sigma,s)
\widehat{f_{k_3}^-}(-\xi-\eta-\sigma,s)\,d\eta d\sigma ds\Big|\\
\lesssim \varepsilon_1^32^{-m/20}2^{-10k_+}.
\end{split}
\end{equation*}
Therefore, for \eqref{bn91} it suffices to prove that
\begin{equation}\label{alo1.5}
\begin{split}
\Big|\int_{\mathbb{R}^2\times[t_1,t_2]}e^{iH(\xi,s)}e^{is\Phi(\xi,\eta,\sigma)}\chi_{k,m}(\eta)\widehat{f_{k_1}^+}(\xi+\eta,s)\widehat{f_{k_2}^+}(\xi+\sigma,s)
\widehat{f_{k_3}^-}(-\xi-\eta-\sigma,s)\,d\eta d\sigma ds\Big|\\
\lesssim \varepsilon_1^32^{-2p_1m}2^{-10k_+}.
\end{split}
\end{equation}

The main observation is that the phase $\Phi$ is weakly elliptic in a suitable sense, more precisely
\begin{equation}\label{alo2}
|\Phi(\xi,\eta,\sigma)|\geq \kappa(\eta):=
\begin{cases}
2^{\min(k,k_2,k_3)/2-100}&\text{ if }k\geq k_1+11,\\
|\eta|2^{-k_1/2-100}&\text{ if }k\leq k_1+10,
\end{cases}
\end{equation}
provided that $|\xi+\eta|\in[2^{k_1-2},2^{k_1+2}]$, $|\xi+\sigma|\in[2^{k_2-2},2^{k_2+2}]$, $|\xi+\eta+\sigma|\in[2^{k_3-2},2^{k_3+2}]$, and $\chi_{k,m}(\eta)\neq 0$. The bound \eqref{alo2} is an easy consequence of the definitions and the assumptions \eqref{alo1}.

We integrate by parts in $s$ to conclude that the integral in the left-hand side of \eqref{alo1.5} is dominated by
\begin{equation*}
C[C^0(\xi)+\sum_{j=1}^2C_j(\xi)]
\end{equation*}
where, with $\dot{H}=\partial_sH$ as before,
\begin{equation*}
\begin{split}
C^0(\xi):=&\int_{t_1}^{t_2}\Big|\int_{\mathbb{R}^2}e^{is\Phi(\xi,\eta,\sigma)}\\
&\times\frac{d}{ds}\Big[\frac{\chi_{k,m}(\eta)}{\Phi(\xi,\eta,\sigma)+\dot{H}(\xi,s)}\widehat{f_{k_1}^{+}}(\xi+\eta,s)\widehat{f_{k_2}^{+}}(\xi+\sigma,s)
\widehat{f_{k_3}^{-}}(-\xi-\eta-\sigma,s)\Big]\,d\eta d\sigma\Big| ds,\\
\end{split}
\end{equation*}
and, for $j\in\{1,2\}$, 
\begin{equation*}
\begin{split}
C_j(\xi):=&\Big|\int_{\mathbb{R}^2}e^{it_j\Phi(\xi,\eta,\sigma)}\frac{\chi_{k,m}(\eta)}{\Phi(\xi,\eta,\sigma)+\dot{H}(\xi,t_j)}\widehat{f_{k_1}^{+}}(\xi+\eta,t_j)\widehat{f_{k_2}^{+}}(\xi+\sigma,t_j)
\widehat{f_{k_3}^{-}}(-\xi-\eta-\sigma,t_j)\,d\eta d\sigma\Big|.
\end{split}
\end{equation*}
Using only the $L^\infty$ bound in \eqref{bn13} and the assumptions \eqref{alo1}, we estimate, for $j\in\{1,2\}$,
\begin{equation*}
C_j(\xi)\lesssim \varepsilon_1^3\kappa^{-1}2^{k_1}2^{\min(k_2,k_3)}2^{-10\max(k_2,k_3,0)}\lesssim \varepsilon_1^32^{-m/4}2^{-10k_+}.
\end{equation*}
Letting
\begin{equation}\label{alo4} m_6(\eta,\sigma):=\frac{\chi_{k,m}(\eta)}{\Phi(\xi,\eta,\sigma)+\dot{H}(\xi,s)}\varphi_{[k_1-1,k_1+1]}(\xi+\eta)\varphi_{[k_2-1,k_2+1]}(\xi+\sigma)\varphi_{[k_3-1,k_3+1]}(\xi+\eta+\sigma),
\end{equation}
for \eqref{alo1.5} it suffices to prove that, for any $s\in[t_1,t_2]$,
\begin{equation}\label{alo5}
\begin{split}
\Big|\int_{\mathbb{R}^2}e^{is\Phi(\xi,\eta,\sigma)}\frac{d}{ds}\big[m_6(\eta,\sigma)\widehat{f_{k_1}^{+}}(\xi+\eta,s)\widehat{f_{k_2}^{+}}(\xi+\sigma,s)
\widehat{f_{k_3}^{-}}(-\xi-\eta-\sigma,s)\big]\,d\eta d\sigma\Big|\\
\lesssim 2^{-m}\varepsilon_1^32^{-2p_1m}2^{-10k_+}.
\end{split}
\end{equation}

Expanding the $d/ds$ derivative, the left-hand side of \eqref{alo5} is dominated by
\begin{equation*}
C[C^0_0(\xi,s)+C^0_1(\xi,s)+C^0_2(\xi,s)+C^0_3(\xi,s)],
\end{equation*}
where
\begin{equation}\label{alo6}
\begin{split} &C^0_0(\xi,s):=\int_{\mathbb{R}^2}\Big|\frac{(\partial_s\dot{H})(\xi,s)\chi_{k,m}(\eta)}{(\Phi(\xi,\eta,\sigma)+\dot{H}(\xi,s))^2}\widehat{f_{k_1}^+}(\xi+\eta,s)\widehat{f_{k_2}^+}(\xi+\sigma,s)
\widehat{f_{k_3}^-}(-\xi-\eta-\sigma,s)\Big|\,d\eta d\sigma,\\
&C^0_1(\xi,s):=\int_{\mathbb{R}^2}\big|m_6(\eta,\sigma)(\partial_s\widehat{f_{k_1}^+})(\xi+\eta,s)\widehat{f_{k_2}^+}(\xi+\sigma,s)
\widehat{f_{k_3}^-}(-\xi-\eta-\sigma,s)\big|\,d\eta d\sigma,\\
&C^0_2(\xi,s):=\Big|\int_{\mathbb{R}^2}e^{is\Phi(\xi,\eta,\sigma)}m_6(\eta,\sigma)\widehat{f_{k_1}^+}(\xi+\eta,s)(\partial_s\widehat{f_{k_2}^+})(\xi+\sigma,s)
\widehat{f_{k_3}^-}(-\xi-\eta-\sigma,s)\,d\eta d\sigma\Big|,\\
&C^0_3(\xi,s):=\Big|\int_{\mathbb{R}^2}e^{is\Phi(\xi,\eta,\sigma)}m_6(\eta,\sigma)\widehat{f_{k_1}^+}(\xi+\eta,s)\widehat{f_{k_2}^+}(\xi+\sigma,s)
(\partial_s\widehat{f_{k_3}^-})(-\xi-\eta-\sigma,s)\,d\eta d\sigma\Big|.
\end{split}
\end{equation}

Using \eqref{bn97.15}, the $L^\infty$ bound in \eqref{bn13}, and \eqref{alo2}, we have
\begin{equation*}
\sup_{s\in[t_1,t_2]}C^0_0(\xi,s)\lesssim \varepsilon_1^32^{-5m/4}2^{-10k_+}.
\end{equation*}
Also, using \eqref{touse4.1}, \eqref{bn13}, and \eqref{alo2},
\begin{equation*}
\sup_{s\in[t_1,t_2]}C^0_1(\xi,s)\lesssim (2^{11m/10}2^{k_1/2})\varepsilon_1^32^{-10k_+}2^{3p_0m}2^{-m/2}(2^{k_1/2}+2^{-m/2})2^{k_1}\lesssim \varepsilon_1^32^{-10k_+}2^{-7m/6}.
\end{equation*}
To estimate the remaining integrals we use Lemma \ref{touse}. Using \eqref{alo1}, \eqref{alo2}, \eqref{bn93.5} and integration by parts it is easy to see that, for any $s\in[t_1,t_2]$,
\begin{equation*}
\|\mathcal{F}^{-1}(m_6)\|_{L^1}\lesssim (2^{11m/10}2^{k_1/2})2^{p_0m}.
\end{equation*}
Therefore, using also \eqref{bn13}, \eqref{bn13.5}, and \eqref{touse4},
\begin{equation*}
\sup_{s\in[t_1,t_2]}[C^0_2(\xi,s)+C^0_3(\xi,s)]\lesssim 2^{11m/10}2^{k_1/2}2^{p_0m}\cdot\varepsilon_12^{k_1/2}\cdot\varepsilon_12^{-m/2}\cdot \varepsilon_12^{(3p_0-1)m}\lesssim \varepsilon_1^32^{-10k_+}2^{-11m/10}.
\end{equation*}
The desired bound \eqref{alo5} follows, which completes the proof of the lemma.
\end{proof}

\begin{lem}\label{bb14}
The bounds \eqref{bn21} hold provided that
\begin{equation}\label{bn140}
\min(k_1,k_2,k_3)+\mathrm{med}(k_1,k_2,k_3)\geq -51m/50.
\end{equation}
\end{lem}

\begin{proof}[Proof of Lemma \ref{bb14}] After changes of variables, it suffices to prove that
\begin{equation}\label{bn141}
\begin{split}
\Big|\int_{\mathbb{R}^2\times[t_1,t_2]}e^{iH(\xi,s)}e^{is\Phi^{\iota_1,\iota_2,\iota_3}(\xi,\eta,\sigma)}\widehat{f_{k_1}^{\iota_1}}(\xi+\eta,s)\widehat{f_{k_2}^{\iota_2}}(\xi+\sigma,s)
\widehat{f_{k_3}^{\iota_3}}(-\xi-\eta-\sigma,s)\,d\eta d\sigma ds\Big|\\
\lesssim \varepsilon_1^32^{-2p_1m}2^{-10k_+},
\end{split}
\end{equation}
where $(\iota_1,\iota_2,\iota_3)\in\{(+,+,+),(+,-,-),(-,-,-)\}$ and 
\begin{equation*}
\begin{split}
&\Phi^{\iota_1,\iota_2,\iota_3}(\xi,\eta,\sigma)=\Lambda(\xi)-\iota_1\Lambda(\xi+\eta)-\iota_2\Lambda(\xi+\sigma)-\iota_3\Lambda(\xi+\eta+\sigma),\\
&H(\xi,s)=\frac{2c_0}{\pi}|\xi|^{3/2}\int_0^s|\widehat{f}(\xi,r)|^2 \frac{dr}{r+1}.
\end{split}
\end{equation*}

The main observation is that the phases $\Phi^{\iota_1,\iota_2,\iota_3}$ are elliptic, i.e.
\begin{equation}\label{bn142}
|\Phi^{\iota_1,\iota_2,\iota_3}(\xi,\eta,\sigma)|\geq 2^{\mathrm{med}(k_1,k_2,k_3)/2-100},
\end{equation}
provided that $|\xi+\eta|\in[2^{k_1-2},2^{k_1+2}]$, $|\xi+\sigma|\in[2^{k_2-2},2^{k_2+2}]$, $|\xi+\eta+\sigma|\in[2^{k_3-2},2^{k_3+2}]$, and $(\iota_1,\iota_2,\iota_3)\in\{(+,+,+),(+,-,-),(-,-,-)\}$. Letting $\dot{H}(\xi,s)=(\partial_sH)(\xi,s)$, we notice that
\begin{equation}\label{bn143}
|\dot{H}(\xi,s)|\lesssim \varepsilon_1^22^{3k/2}2^{-20k_+}2^{-m}.
\end{equation}

As in the proof of Lemma \ref{bb13}, we integrate by parts in $s$ to conclude that the integral in the left-hand side of \eqref{bn141} is dominated by
\begin{equation*}
D^0(\xi)+\sum_{j=1}^2D_j(\xi)
\end{equation*}
where
\begin{equation}\label{bn144}
\begin{split}
D^0(\xi):=&\int_{t_1}^{t_2}\Big|\int_{\mathbb{R}^2}e^{is\Phi^{\iota_1,\iota_2,\iota_3}(\xi,\eta,\sigma)}\\
&\times\frac{d}{ds}\Big[\frac{1}{\Phi^{\iota_1,\iota_2,\iota_3}(\xi,\eta,\sigma)+\dot{H}(\xi,s)}\widehat{f_{k_1}^{\iota_1}}(\xi+\eta,s)\widehat{f_{k_2}^{\iota_2}}(\xi+\sigma,s)
\widehat{f_{k_3}^{\iota_3}}(-\xi-\eta-\sigma,s)\Big]\,d\eta d\sigma\Big| ds,\\
\end{split}
\end{equation}
and, for $j\in\{1,2\}$, 
\begin{equation}\label{bn145}
\begin{split}
D_j(\xi):=&\Big|\int_{\mathbb{R}^2}e^{it_j\Phi^{\iota_1,\iota_2,\iota_3}(\xi,\eta,\sigma)}\\
&\times\frac{1}{\Phi^{\iota_1,\iota_2,\iota_3}(\xi,\eta,\sigma)+\dot{H}(\xi,t_j)}\widehat{f_{k_1}^{\iota_1}}(\xi+\eta,t_j)\widehat{f_{k_2}^{\iota_2}}(\xi+\sigma,t_j)
\widehat{f_{k_3}^{\iota_3}}(-\xi-\eta-\sigma,t_j)\,d\eta d\sigma\Big|.
\end{split}
\end{equation}

Let
\begin{equation*}
\begin{split}
 m_7(\eta,\sigma):=&\frac{1}{\Phi^{\iota_1,\iota_2,\iota_3}(\xi,\eta,\sigma)+\dot{H}(\xi,s)}\\
 &\times \varphi_{[k_1-1,k_1+1]}(\xi+\eta)\varphi_{[k_2-1,k_2+1]}(\xi+\sigma)\varphi_{[k_3-1,k_3+1]}(\xi+\eta+\sigma).
\end{split}
\end{equation*}
Using \eqref{bn140}, \eqref{bn142}, \eqref{bn143} and integration by parts it is easy to see that, for any $s\in[t_1,t_2]$,
\begin{equation}\label{bn146}
\|\mathcal{F}^{-1}(m_7)\|_{L^1}\lesssim 2^{-\mathrm{med}(k_1,k_2,k_3)/2}.
\end{equation}
Therefore we can apply Lemma \ref{touse} with
\begin{equation*}
\begin{split}
&\widehat{f}(\eta):=e^{-it_j\iota_1\Lambda(\xi+\eta)}\widehat{f_{k_1}^{\iota_1}}(\xi+\eta,t_j),\\
&\widehat{g}(\sigma):=e^{-it_j\iota_2\Lambda(\xi+\sigma)}\widehat{f_{k_2}^{\iota_2}}(\xi+\sigma,t_j),\\
&\widehat{h}(\theta):=e^{-it_j\iota_3\Lambda(\xi-\theta)}\widehat{f_{k_3}^{\iota_3}}(-\xi+\theta,t_j),
\end{split}
\end{equation*}
and then use \eqref{bn13} and \eqref{bn13.5}, to conclude that
\begin{equation}\label{bn147}
\begin{split}
D_j(\xi)&\lesssim 2^{-\mathrm{med}(k_1,k_2,k_3)/2}\cdot\varepsilon_1 2^{-m/2}\cdot \varepsilon_1 2^{\min(k_1,k_2,k_3)/2}\cdot \varepsilon_1 2^{\mathrm{med}(k_1,k_2,k_3)/2}\\
&\lesssim \varepsilon_1^32^{\min(k_1,k_2,k_3)/2}2^{-m/2}.
\end{split}
\end{equation}
for $j\in\{1,2\}$.

In addition, we estimate
\begin{equation}\label{bn148}
\begin{split}
&D^0(\xi)\lesssim 2^m\sup_{s\in[t_1,t_2]}[D^0_0(\xi,s)+D^0_1(\xi,s)+D^0_2(\xi,s)+D^0_3(\xi,s)],\\  &D^0_0(\xi,s):=\int_{\mathbb{R}^2}\Big|\frac{(\partial_s\dot{H})(\xi,s)}{(\Phi^{\iota_1,\iota_2,\iota_3}(\xi,\eta,\sigma)+\dot{H}(\xi,s))^2}\widehat{f_{k_1}^{\iota_1}}(\xi+\eta,s)\widehat{f_{k_2}^{\iota_2}}(\xi+\sigma,s)
\widehat{f_{k_3}^{\iota_3}}(-\xi-\eta-\sigma,s)\Big|\,d\eta d\sigma,\\
&D^0_1(\xi,s):=\Big|\int_{\mathbb{R}^2}e^{is\Phi^{\iota_1,\iota_2,\iota_3}(\xi,\eta,\sigma)}m_7(\eta,\sigma)(\partial_s\widehat{f_{k_1}^{\iota_1}})(\xi+\eta,s)\widehat{f_{k_2}^{\iota_2}}(\xi+\sigma,s)
\widehat{f_{k_3}^{\iota_3}}(-\xi-\eta-\sigma,s)\,d\eta d\sigma\Big|,\\
&D^0_2(\xi,s):=\Big|\int_{\mathbb{R}^2}e^{is\Phi^{\iota_1,\iota_2,\iota_3}(\xi,\eta,\sigma)}m_7(\eta,\sigma)\widehat{f_{k_1}^{\iota_1}}(\xi+\eta,s)(\partial_s\widehat{f_{k_2}^{\iota_2}})(\xi+\sigma,s)
\widehat{f_{k_3}^{\iota_3}}(-\xi-\eta-\sigma,s)\,d\eta d\sigma\Big|,\\
&D^0_3(\xi,s):=\Big|\int_{\mathbb{R}^2}e^{is\Phi^{\iota_1,\iota_2,\iota_3}(\xi,\eta,\sigma)}m_7(\eta,\sigma)\widehat{f_{k_1}^{\iota_1}}(\xi+\eta,s)\widehat{f_{k_2}^{\iota_2}}(\xi+\sigma,s)
(\partial_s\widehat{f_{k_3}^{\iota_3}})(-\xi-\eta-\sigma,s)\,d\eta d\sigma\Big|.
\end{split}
\end{equation}

As before, we combine Lemma \ref{touse}, \eqref{bn146}, and the bounds \eqref{bn13}, \eqref{bn13.5}, and \eqref{touse4} to conclude that
\begin{equation}\label{bn148.1}
\begin{split}
\sup_{s\in[t_1,t_2]}[D^0_1(\xi,s)+D^0_2(\xi,s)+D^0_3(\xi,s)]&\lesssim 2^{-\mathrm{med}(k_1,k_2,k_3)/2}\cdot\varepsilon_1 2^{-m/2}\cdot \varepsilon_12^{\mathrm{med}(k_1,k_2,k_3)/2}\cdot \varepsilon_1 2^{(3p_0-1)m}\\
&\lesssim \varepsilon_1^32^{(3p_0-3/2)m}.
\end{split}
\end{equation}
In addition, using the $L^\infty$ bound in \eqref{touse4} and the definition of the function $H$, we have
\begin{equation*}
\sup_{s\in[t_1,t_2]}|(\partial_s\dot{H})(\xi,s)|\lesssim \varepsilon_1^22^{3k/2}2^{-20k_+}2^{(3p_0-3/2)m}.
\end{equation*}
Therefore
\begin{equation}\label{bn148.2}
\sup_{s\in[t_1,t_2]}D^0_0(\xi,s)\lesssim \varepsilon_1^32^{3k/2}2^{-20k_+}2^{(3p_0-3/2)m}\cdot 2^{\min(k_1,k_2,k_3)}.
\end{equation}

We combine now \eqref{bn147}--\eqref{bn148.2} to conclude that the left-hand side of \eqref{bn141} is dominated by $C\varepsilon_1^32^{-m/3}$, 
which is clearly stronger than the desired bound \eqref{bn91}. This completes the proof of the lemma.
\end{proof}


\begin{thebibliography}{100}

\bibitem{Barab}
 Barab J. E.
\newblock Non-existence of asymptotically free solutions for nonlinear \S equation.
\newblock {\em Journal of Math. Phys.}, 25 (1984), no. 11, 3270-3273.




\bibitem{Ch} Christodoulou D. 
\newblock{Global solutions of nonlinear hyperbolic equations for small initial data.} 
\newblock{\em Comm. Pure Appl. Math.} 39 (1986), no. 2, 267-282.










\bibitem{CW}
Craig, W. and Worfolk, P.
\newblock An integrable normal form for water waves in infinite depth.
\newblock {\em Phys. D} 84 (1995), no. 3-4, 513--531



\bibitem{Craig}
Craig, W.
\newblock Birkhoff normal forms for water waves. {\em Mathematical problems in the theory of water waves (Luminy, 1995)}, 57-74.
\newblock {\em Contemp. Math.}, 200, Amer. Math. Soc., Providence, RI, 1996. 




\bibitem{DelortKG1d} 
Delort, J.M.
\newblock Existence globale et comportement asymptotique pour l' \'{e}quation de Klein-Gordon quasi-lin\'{e}aire \`{a} donn\'{e}es
petites en dimension 1. 
\newblock {\em Ann. Sci. \'{E}cole Norm. Sup.} 34 (2001) 1-61.
\newblock Erratum: ``Global existence and asymptotic behavior for the quasilinear Klein-Gordon equation with small data in dimension 1''
\newblock {\em Ann. Sci. \'{E}cole Norm. Sup.} (4) 39 (2006), no. 2, 335-345







\bibitem{GMS1}
Germain, P., Masmoudi, N. and Shatah, J.
\newblock {Global solutions for quadratic Schr\"odinger equations in dimension 3}.
\newblock {\em Int. Math. Res. Not.} (2009), no. 3, 414-432.



\bibitem{GMS2}
Germain P., Masmoudi, N. and Shatah, J.
\newblock {Global solutions for the gravity surface water waves equation in dimension 3}.
\newblock {\em Annals of Math.}, to appear.









\bibitem{GOVNLS}
Ginibre, J., Ozawa, T. and Velo, G.
\newblock On the existence of the wave operators for a class of nonlinear \S equations.
\newblock {\em Ann. Inst. H. Poincar\'{e} Phys. Th\'{e}or.} 60 (1994), no. 2, 211-239



\bibitem{GuoHuo}
Guo, B. and Huo, Z.
\newblock Global well-posedness for the fractional nonlinear Schr\"odinger equation.
\newblock {\em Comm. PDE} 36 (2011), no. 2, 247-255.



\bibitem{HN}
Hayashi, N. and Naumkin, P.I.
\newblock Asymptotics for large time of solutions to the nonlinear Schr\"odinger and Hartree equations.
\newblock {\em Amer. J. Math.}, 120 (1998), no. 2, 369-389.








\bibitem{HNBO}
Hayashi, N. and Naumkin, P.
\newblock Large time asymptotics of solutions to the generalized Benjamin-Ono equation.
\newblock {\em Trans. Amer. Math. Soc.}, 351 (1999), no. 1, 109-130.




\bibitem{HNKdV}
Hayashi, N. and Naumkin, P.
\newblock Large time behavior of solutions for the modified Korteweg-de Vries equation.
\newblock {\em Int. Math. Res. Not.}, (1999), no. 8, 395-418.










\bibitem{HNcubicNLSodd}
Hayashi, N. and Naumkin, P.
\newblock Asymptotics of odd solutions for cubic nonlinear \S equations.
\newblock {\em J. Differential Equations}, 246 (2009), no. 4, 1703-1722.







\bibitem{Joh} 
John F. 
\newblock{Blow-up of solutions of nonlinear wave equations in three space dimensions} 
\newblock{\em Manuscripta Math.}, 28 (1979), no. 1-3, 235-268.




\bibitem{KP}
Kato, J. and Pusateri, F.
\newblock {A new proof of long range scattering for critical nonlinear
  Schr\"odinger equations}.
\newblock {\em Diff. Int. Equations}, 24 (2011), no. 9-10, 923-940.



\bibitem{Kl}
Klainerman, S.
\newblock Uniform decay estimates and the Lorentz invariance of the classical wave equation. 
\newblock {\em  Comm. Pure Appl. Math.}, 38 (1985), no. 3, 321-332.




\bibitem{Kl2} Klainerman, S. 
\newblock{The null condition and global existence to nonlinear wave equations. 
Nonlinear systems of partial differential equations in applied mathematics, Part 1 (Santa Fe, N.M., 1984).} 
\newblock{\em Lectures in Appl. Math.} 23, 293-326,  Amer. Math. Soc., Providence, RI, 1986.



\bibitem{Lannes}
Lannes, D.
\newblock Well-posedness of the water waves equations.
\newblock {\em Journal of Amer. Math. Soc.} 18 (2005), no. 3, 605-654.





\bibitem{Laskin2}
Laskin, N.
\newblock Fractional \S equation.
\newblock {\em Phys. Rev. E} 66 (2002), no. 5, 056108, 7 pp. 



\bibitem{Lindblad}
Lindblad, H.
\newblock Well-posedness for the motion of an incompressible liquid with free surface boundary.
\newblock {\em Ann. of Math.}, 162 (2005), no. 1, 109-194.



\bibitem{NakaNLS}
Nakanishi, K. 
\newblock Asymptotically-free solutions for the short-range nonlinear \S equation. 
\newblock {\em SIAM J. Math. Anal.} 32 (2001), no. 6, 1265-1271.



\bibitem{ozawa}
Ozawa, T.
\newblock {Long range scattering for nonlinear Schr\"odinger equations in one space dimension}.
\newblock {\em Comm. Math. Phys.}, 139 (1991), no. 3, 479-493.



\bibitem{Sh} Shatah, J. 
\newblock Normal forms and quadratic nonlinear Klein-Gordon equations. 
\newblock {\em Comm. Pure Appl. Math.}, 38 (1985), no.5, 685-696.




\bibitem{ShZ3}
Shatah, J. and Zeng, C.
\newblock Local well-posedness for the fluid interface problem.
\newblock {\em Arch. Ration. Mech. Anal.} 199 (2011), no. 2, 653-705.




\bibitem{SulemBook}
Sulem C. and Sulem. 
\newblock Self-focusing and wave collapse.
\newblock {\em Book}, Springer 1993.



\bibitem{Wuloc1}
Wu, S. 
\newblock{Well-posedness in Sobolev spaces of the full water wave problem in 2-D}. 
\newblock{\em Invent. Math.} 130 (1997), no. 1, 39-72.


\bibitem{Wuloc2}
Wu, S. 
\newblock{Well-posedness in Sobolev spaces of the full water wave problem in 3-D}. 
\newblock {\em J. Amer. Math. Soc.} 12 (1999), no. 2, 445-495.


\bibitem{WuAG}
Wu, S.
\newblock {Almost global wellposedness of the 2-D full water wave problem}.
\newblock {\em Invent. Math.}, 177 (2009), no. 1, 45-135.



\bibitem{Wu3DWW}
Wu, S.
\newblock {Global wellposedness of the 3-D full water wave problem}.
\newblock {\em Invent. Math.}, 184 (2011), no. 1, 125-220.



\bibitem{WuNLS}
Wu, S.
\newblock {A rigorous justification of the modulation approximation to the 2D full water wave problem}.
\newblock {\em arXiv:1101.0545}, 2011.


\end{thebibliography}
\end{document}